\documentclass[11pt]{article}

\usepackage{epsfig}
\usepackage{graphicx}
\usepackage{amsbsy}
\usepackage{amsmath}
\usepackage{amsfonts}
\usepackage{amssymb}
\usepackage{textcomp}
\usepackage{hyperref}
\usepackage{aliascnt}

\newcommand{\mcm}[3]{\newcommand{#1}[#2]{{\ensuremath{#3}}}} 

\mcm{\tuple}{1}{\langle #1 \rangle}
\mcm{\name}{1}{\ulcorner #1 \urcorner}
\mcm{\Nbb}{0}{\mathbb{N}}
\mcm{\Zbb}{0}{\mathbb{Z}}
\mcm{\Rbb}{0}{\mathbb{R}}
\mcm{\Cbb}{0}{\mathbb{C}}
\mcm{\Qbb}{0}{\mathbb{Q}}
\mcm{\Acal}{0}{\cal A}
\mcm{\Bcal}{0}{\cal B}
\mcm{\Ccal}{0}{\cal C}
\mcm{\Dcal}{0}{\cal D}
\mcm{\Ecal}{0}{\cal E}
\mcm{\Fcal}{0}{\cal F}
\mcm{\Gcal}{0}{\cal G}
\mcm{\Hcal}{0}{\cal H}
\mcm{\Ical}{0}{\cal I}
\mcm{\Jcal}{0}{\cal J}
\mcm{\Kcal}{0}{\cal K}
\mcm{\Lcal}{0}{\cal L}
\mcm{\Mcal}{0}{\cal M}
\mcm{\Ncal}{0}{\cal N}
\mcm{\Ocal}{0}{{\cal O}}
\mcm{\Pcal}{0}{{\cal P}}
\mcm{\Qcal}{0}{{\cal Q}}
\mcm{\Rcal}{0}{{\cal R}}
\mcm{\Scal}{0}{{\cal S}}
\mcm{\Tcal}{0}{{\cal T}}
\mcm{\Ucal}{0}{{\cal U}}
\mcm{\Vcal}{0}{{\cal V}}
\mcm{\Wcal}{0}{{\cal W}}
\mcm{\Xcal}{0}{{\cal X}}
\mcm{\Ycal}{0}{{\cal Y}}
\mcm{\Zcal}{0}{{\cal Z}}
\mcm{\Mfrak}{0}{\mathfrak M}

\mcm{\restric}{0}{\upharpoonright}
\mcm{\upset}{0}{\uparrow}
\mcm{\onto}{0}{\twoheadrightarrow}
\mcm{\smallNbb}{0}{{\small \mathbb{N}}}
\DeclareMathOperator{\preop}{op}
\mcm{\op}{0}{^{\preop}}

\newcommand{\se}{\subseteq}
{\begin{array}{c}
\setlength{\unitlength}{1em}}%
{\end{array}}

\usepackage{amsthm}

\newcommand{\theoremize}[2]{\newaliascnt{#1}{thm} \newtheorem{#1}[#1]{#2} \aliascntresetthe{#1}}

\theoremstyle{plain}
\newtheorem{thm}{Theorem}[section]
\theoremize{lem}{Lemma}
\theoremize{skolem}{Skolem}
\theoremize{fact}{Fact}
\theoremize{sublem}{Sublemma}
\theoremize{claim}{Claim}
\theoremize{obs}{Observation}
\theoremize{prop}{Proposition}
\theoremize{cor}{Corollary}
\theoremize{que}{Question}
\theoremize{oque}{Open Question}
\theoremize{con}{Conjecture}

\theoremstyle{definition}
\theoremize{dfn}{Definition}
\theoremize{rem}{Remark}
\theoremize{eg}{Example}
\theoremize{exercise}{Exercise}
\theoremstyle{plain}

\usepackage{verbatim}
\usepackage{enumerate}
\usepackage[all]{xy}

\usepackage{subfig}

\title{Embedding simply connected \\ 2-complexes in 3-space\\ \Large I. A Kuratowski-type 
characterisation}

\author{Johannes Carmesin
\medskip 
\\
  {University of Birmingham}
}
\newcommand{\sm}{\setminus}

\DeclareMathOperator{\Sbb}{\mathbb{S}}

\newcommand{\Sthree}{$\Sbb^3$}

\mcm{\Fbb}{0}{\mathbb{F}}

\usepackage{xr}
\begin{document}

\maketitle

\begin{abstract}
We characterise the embeddability of simply connected locally 3-connected 2-dimensional simplicial 
complexes in 3-space in a way analogous to Kuratowski's characterisation of graph planarity, by 
excluded minors. This answers questions of Lov\'asz, Pardon and Wagner.
\end{abstract}

\section{Introduction}

In 1930, Kuratowski proved that a graph can be embedded in the plane if and only if it has none of 
the two non-planar graphs $K_5$ or $K_{3,3}$ as a minor\footnote{A \emph{minor} of a graph is 
obtained by deleting or contracting edges.}. The main result of this paper may be regarded as a 
3-dimensional analogue of this theorem. 

\vspace{.3cm}

Kuratowski's theorem gives a way how embeddings in the plane could be understood through the 
minor relation. A far reaching extension of Kuratowski's theorem is the Robertson-Seymour theorem 
\cite{MR2099147}. Any minor closed class of graphs is characterised by the list of minor-minimal 
graphs not 
in the class. This theorem says that this list always must be finite. 
The methods developed to prove this theorem are nowadays used in many results in the 
area of structural graph theory \cite{DiestelBookCurrent} -- 
and beyond; recently Geelen, Gerards and Whittle extended the Robertson-Seymour theorem  
to representable matroids by 
proving  Rota's conjecture \cite{solving_Rota}. 
Very roughly, the Robertson-Seymour structure theorem establishes a correspondence between minor 
closed classes of graphs and classes of graphs almost embeddable in 2-dimensional surfaces.

In his survey on the Graph Minor project of Robertson and Seymour \cite{lovasz_gm_survey}, in 2006 
Lov\'asz asked 
whether there is a 
meaningful analogue of the minor relation in three dimensions. Clearly, every graph can be 
embedded in 3-space\footnote{Indeed, embed the vertices in general position and embed the edges 
as straight lines. }.

One approach towards this question is to restrict the embeddings in question, and just consider so 
called linkless embeddings of graphs, see \cite{linkless_emb_survey} for a survey. Instead of 
restricting embeddings, one could also 
put some additional structure on the graphs in question. Indeed, Wagner  asked how an analogue of 
the minor relation could be defined on general simplicial complexes \cite{Wagner_minor}. 

Unlike in higher dimensions, a 2-dimensional simplicial complex has a topological embedding in 
3-space if and only if it has a piece-wise linear embedding if and only if it has a differential 
embedding \cite{{Bin59},{Hatcher3notes},{moise},{Pap43}}. In 
\cite{mstw_3sphere_decidable}, Matou\v sek, Sedgewick, Tancer and Wagner 
proved that 
the embedding problem for 2-dimensional simplicial complexes in 3-space is 
decidable. In August 2017,
de Mesmay, Rieck, Sedgwick and Tancer complemented this result by showing that this problem is 
NP-hard \cite{s3np_hard}. 

This might suggest that if we would like to get a structural characterisation of embeddability, we 
should work inside a subclass of 2-dimensional simplicial complexes. And in fact such 
questions have 
been asked: in 2011 at the internet forum `MathsOverflow' Pardon\footnote{John Pardon confirmed in 
private communication that he asked that question as the user `John Pardon'.} asked 
whether there are necessary and sufficient conditions for when contractible 2-dimensional 
simplicial complexes embed in 3-space. The \emph{link graph} at a vertex $v$ of a 
simplicial complex 
is the incidence graph between edges and faces incident with $v$.
He notes that if embeddable the link graph at any vertex must be planar. This leads to obstructions 
for embeddability such as the cone over the complete graph $K_5$, see 
\autoref{cK5}. -- But there are different 
obstructions of a more global character, see \autoref{8obs}. All their link graphs are planar -- 
yet they are not embeddable.
   \begin{figure} [htpb]   
\begin{center}
   	  \includegraphics[height=4cm]{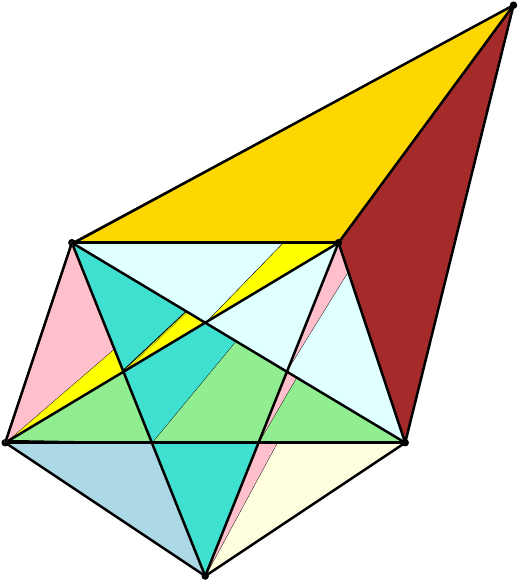}
   	  \caption{The cone over $K_5$. 
Similarly as the graph $K_5$ does not embed in 2-space, the cone over 
$K_5$ does not embed in 3-space. }\label{cK5}
\end{center}
   \end{figure}

   \begin{figure} [htpb]   
\begin{center}
   	  \includegraphics[height=4cm]{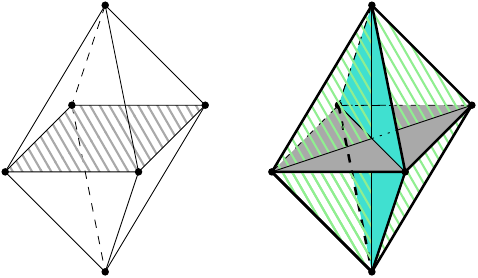}
   	  \caption{The octahedron obstruction, depicted on the right, is obtained from the 
octahedron with its eight triangular faces by 
adding 3 more faces of size 4 orthogonal to the three axis. If we add just one of these 
4-faces 
to the octahedron, the resulting 2-complex is embeddable as illustrated on the left. A 
second 4-face could be added on the outside of that depicted embedding. However, it can 
be shown that the octahedron with all three 4-faces is not embeddable.}\label{8obs}
\end{center}
   \end{figure}

\vspace{.3cm}

Addressing these questions, we introduce an analogue of the minor relation 
for 2-complexes and we use it to prove a 3-dimensional analogue of 
Kuratowski's theorem characterising when simply connected 2-dimensional simplicial complexes 
(topologically) embed
in 3-space.

\begin{figure}
\begin{center}

    \begin{tabular}{ | l | l  | l |}
    \hline
    {} & delete & contract  \\ \hline
    edge & {} &{} \\ \hline
    face & {}&{} \\ \hline
    \end{tabular} 
    \caption{For each of the four corners of the above diagram we have one space minor 
operation.}\label{4ops}
\end{center}
\end{figure}

More precisely, a \emph{space minor} of a 2-complex is obtained by successively 
deleting or contracting edges or faces, and splitting vertices. 
See \autoref{4ops} and \autoref{fig:space_minor2}. The precise details of these definitions are 
given in \autoref{sec:space}; for example contraction of edges is only allowed for edges that 
are not loops\footnote{\emph{Loops} are edges that have only a single endvertex. While 
contraction of edges that are not loops clearly preserves embeddability in 3-space, for loops this 
is not always the case.} and we only contract faces of size at most two.
   \begin{figure} [htpb]   
\begin{center}
   	  \includegraphics[height=3cm]{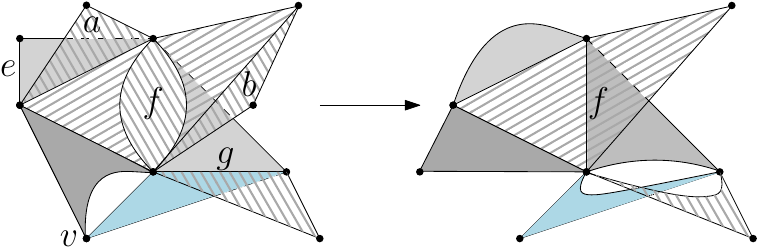}
   	  \caption{The complex on the right is a space minor of the complex on the left. Indeed, 
for that just delete the faces labelled $a$ and $b$, contract the edge $e$ and contract the face 
$f$, and delete the edge $g$ and split the vertex $v$.}\label{fig:space_minor2}
\end{center}
   \end{figure}

It will be quite easy to see that space minors preserve embeddability in 3-space and that this 
relation 
is well-founded. The operations of face deletion and 
face contraction correspond to the minor operations in the dual matroids of simplicial complexes in 
the sense of \cite{3space4}. 

\begin{eg}
 Using space minors, we can understand why the Octahedron Obstruction (\autoref{8obs}) does not 
embed in 3-space. Indeed, we contract an arbitrary face of size three to a single vertex (formally, 
we first contract an edge of that face, then it gets size two. So we can contract it to an edge. 
Then we contract that edge to a vertex). It turns out that the link graph at the new vertex is the 
non-planar graph $K_{3,3}$. Thus this space minor is not embeddable in 3-space. As space minors 
preserve embeddability, we deduce that the Octahedron Obstruction cannot be embeddable. 
\end{eg}

A construction of a simply connected 2-complex that is not embeddable in 3-space and has no space 
minor with a non-planar link graph is described in \autoref{other_construction}.
   \begin{figure} [htpb]   
\begin{center}
   	  \includegraphics[height=3cm]{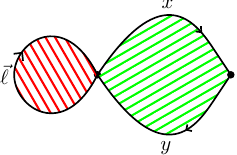}
   	  \caption{We start the construction with the depicted 2-complex consisting of two 
vertices joined by the two parallel edges $x$ and $y$ and a loop $\ell$ attached at one of these 
vertices. It has two faces, one bounded by $\ell$, and the other by $x$ and $y$. Now we add two 
more faces of size three: $xy\protect\overrightarrow{\ell}$ and $xy\protect\overleftarrow{\ell}$. 
They have both the same edge set but they traverse the loop $\ell$ in different directions. It can 
be shown that the resulting 2-complex does not embed in 3-space.}\label{other_construction}
\end{center}
   \end{figure}
The main result of this paper is the following.

\begin{thm}\label{kura_intro}
Let $C$ be a simply connected locally 3-connected 2-dimensional simplicial complex. The following 
are equivalent.
\begin{itemize}
 \item $C$ embeds in 3-space;
 \item $C$ has no space minor from the finite list $\Zcal$.
\end{itemize}
\end{thm}

The finite list $\Zcal$ is defined explicitly in \autoref{sec:gen_cone} below. 
The members of $\Zcal$ are grouped in six natural classes. 
Here a (2-dimensional) simplicial complex is \emph{locally 
3-connected} if all its link graphs are connected and do not contain 
separators of size one or two.
In \cite{3space5}, we extend \autoref{kura_intro} to simplicial complexes that need not be 
locally 
3-connected. For general simplicial complexes, not necessarily simply connected 
ones, the proof implies that a locally 3-connected simplicial complex has an embedding into some 
oriented 3-manifold if and only if it does not have a minor from $\Lcal$. 
\vspace{.3cm}

We are able to extend \autoref{kura_intro} from simply connected 
simplicial complexes to those whose first homology group is trivial.

\begin{thm}\label{kura_intro_hom}
Let $C$ be a locally 3-connected 2-dimensional simplicial complex such that the first homology 
group $H_1(C,\Fbb_p)$ is trivial for some prime $p$. The following are equivalent.
\begin{itemize}
 \item $C$ embeds in 3-space;
 \item $C$ is simply connected and has no space minor from 
the finite list $\Zcal$.
\end{itemize}
\end{thm}

In general there are infinitely many obstructions to embeddability in 3-space. Indeed, the 
following infinite family of obstructions appears in \autoref{kura_intro_hom}.
\begin{eg}\label{q-folded}
Given a natural number $q\geq 2$, the $q$-folded cross cap consists of a single vertex, 
a single edge that is a loop and a single face traversing the edge $q$-times in the same 
direction. It can be shown that $q$-folded cross caps cannot be embedded in 3-space.  
\end{eg}
\noindent A more sophisticated infinite family is constructed in 
\cite{3space4}.

\vspace{.3cm}

{\bf Overview over this series as a whole.}
This paper is the first paper in a series of five paper. In what follows we summarise roughly the 
content of the other four papers \cite{3space2, 3space3, 3space4, 3space5}.
The results of \cite{3space2} give combinatorial characterisations when simplicial complexes embed 
in 3-space, which are used in the proofs of \autoref{kura_intro} and \autoref{kura_intro_hom}. The 
paper \cite{3space2} is self-contained.

As mentioned above, the main result of \cite{3space5} is an extension of \autoref{kura_intro} to 
simply connected simplicial complexes. This relies on the current paper and \cite{3space2}.

The paper \cite{3space3} is purely graph-theoretic and its results are used as a tool in 
\cite{3space4}. 

In \cite{3space4}, we prove an extension of the main theorem of  \cite{3space5} that goes beyond 
the simply connected 
case. And we additionally prove the following. Like Kuratowski's theorem, Whitney's theorem is a  
characterisation of planarity of graphs. 
In \cite{3space4} we prove a 3-dimensional analogue of that theorem.

\vspace{.3cm}

{\bf This paper is organised as follows. }
Most of this paper is concerned with the proof of \autoref{kura_intro_hom}, which implies 
\autoref{kura_intro}.
In \autoref{s_rot}, we introduce `planar rotation systems' and state a theorem of \cite{3space2} 
that 
relates embeddability of simply connected simplicial complexes to existence of planar rotation 
systems. 
In \autoref{sec_vertex_sum} we define the operation of `vertex sums' and use it to study rotation 
systems.
In \autoref{s_constr} we relate the existence of planar rotation systems to a property called 
`local planarity'. In \autoref{sec:marked} we characterise local planarity in terms of finitely 
many obstructions. In \autoref{sec:space} we introduce space minors and prove \autoref{kura_intro} 
and \autoref{kura_intro_hom}.

\vspace{.3cm}

For graphs\footnote{In this paper graphs are allowed to have loops and parallel edges. } we follow 
the notation of \cite{DiestelBookCurrent}. Beyond 
that a \emph{2-complex} is a graph $(V,E)$ together with a set $F$ of closed trails\footnote{A 
\emph{trail} is sequence 
$(e_i|i\leq n)$ of distinct edges such that the endvertex of $e_i$ is the starting vertex of 
$e_{i+1}$ for all $i<n$.  A trail is \emph{closed} if the starting vertex of $e_1$ is equal to 
the endvertex of $e_n$.}, called its 
\emph{faces}. In this paper we follow the convention that each vertex or edge of a 
simplicial complex or a 2-complex is incident with a face. 
The definition of  \emph{link graphs} naturally extends from simplicial complexes to 
2-complexes with the following addition: we add two vertices in the link graph $L(v)$ 
for each loop incident 
with $v$. We add one edge to $L(v)$ for each traversal of a face at $v$.

\section{Rotation systems}\label{s_rot}

Rotation systems of 2-complexes play a central role in our proof of \autoref{kura_intro}. In this 
section we 
introduce them and prove some basic properties of them.

A rotation system of a graph $G$ is a family $(\sigma_v|v\in V(G))$ of cyclic 
orientations\footnote{A \emph{cyclic orientation} is a bijection to an oriented cycle.} 
$\sigma_v$ of the edges incident with the vertices $v$ \cite{MoharThomassen}. The 
orientations $\sigma_v$ 
are called 
\emph{rotators}. Any rotation system of a graph $G$ induces an 
embedding of $G$ in an oriented (2-dimensional) surface $S$. To be precise, we obtain $S$ from $G$ 
by gluing faces onto (the geometric realisation of) $G$ along closed walks of $G$ as follows.
Each directed edge of $G$ is in one of these walks. Here the direction $\vec{a}$ is 
directly before the direction $\vec{b}$ in a face $f$ if the endvertex $v$ of $\vec{a}$ is equal to 
the starting vertex of $\vec{b}$ and $b$ is just after $a$ in the rotator at $v$.
The rotation system is \emph{planar} if that surface $S$ is a disjoint union of 2-spheres. 
Note that if the graph $G$ is connected, then for any rotation system of $G$, also the 
surface $S$ is connected.

A \emph{rotation system of a (directed\footnote{A \emph{directed} 2-complex is a 
2-complex together with a choice of 
direction at each of its edges and a choice of 
orientation at each of its 
faces. All 2-complexes considered in this paper are directed. In order to simplify notation we will 
not always say that explicitly. }) 2-complex} $C$ is a  family $(\sigma_e|e\in E(C))$ of 
cyclic orientations $\sigma_e$ of the faces incident with the edge $e$.
A rotation system of a 2-complex $C$ \emph{induces} a rotation system at each of its link graphs 
$L(v)$ by 
restricting to the edges that are vertices of the link graph $L(v)$; here we take $\sigma(e)$ if 
$e$ 
is directed towards $v$ and the reverse of $\sigma(e)$ otherwise. 

A rotation system of a 2-complex is \emph{planar} if all induced rotation systems of link graphs 
are planar. In \cite{3space2} we prove the following, which we use in the proof of 
\autoref{kura_intro}.

\begin{thm}\label{emb_to_rot}{ \cite[\autoref*{combi_intro}]{3space2}}
A simply connected simplicial complex has an embedding in \Sthree\  if and only if it has a planar 
rotation system. 
\end{thm}

Given a 2-complex $C$, its link graph $L(v)$ is \emph{loop-planar} if it has a planar rotation 
system such that for every loop $\ell$ incident with $v$ the rotators at the two vertices $e_1$ and 
$e_2$
associated to $\ell$ are reverse -- when we apply the following bijection between the edges 
incident with $e_1$ and $e_2$. If $f$ is an edge incident with the vertex $e_1$ whose face of $C$ 
consists 
only of the loop $\ell$, then $f$ is an edge between $e_1$ and $e_2$ and the bijection is identical 
at that edge. If the face $f$ is incident with more edges than $\ell$, it can by assumption 
traverse 
$\ell$ only once. So there are precisely two edges for that traversal, one incident with $e_1$, the 
other with $e_2$. These two edges are in bijection.

A 2-complex $C$ is \emph{locally planar} if all its link graphs are loop-planar. Clearly, a 
2-complex 
that has a planar rotation system is locally planar. However, the converse is not true. 

Let $C=(V,E,F)$ be a 2-complex and let $x$ be a non-loop edge of $C$, the 2-complex obtained from 
$C$ by \emph{contracting $x$} (denoted by $C/x$) is obtained from $C$ by identifying the two 
endvertices of $x$, deleting $x$ from all faces and then deleting $x$, formally: 
$C/x=((V,E)/x, \{f-x|f\in F\})$. 

Let $C$ be a 2-complex and $x$ be a non-loop edge of $C$, and $\Sigma=(\sigma_e|e\in E(C))$ be a 
rotation system of $C$. The \emph{induced} rotation system of $C/x$ is $\Sigma_x=(\sigma_e|e\in 
E(C)-x)$. This is well-defined as the incidence relation between edges of $C/x$ and faces is the 
same as in $C$. 
Planarity of rotation systems is preserved under contractions:

\begin{lem}\label{contr_pres_planar}
If $\Sigma$ is planar, then $\Sigma_x$ is planar. 

Conversely, for any planar rotation system $\Sigma'$ of $C/x$, if the non-loop edge $x$ is not 
a cutvertex of any of the two link graphs at its endvertices, there is a planar rotation system of 
$C$ 
inducing $\Sigma'$.\footnote{This lemma is proved in 
\autoref{sec_vertex_sum}. }
\end{lem}

Hence the class of 2-complexes that have planar rotation systems is closed under contractions. As 
noted above it contains the class of locally planar 2-complexes, which is clearly not closed under 
contractions. However, if we close the later class under contractions, then 
they do agree -- in the locally 3-connected case as follows. 

\begin{lem}\label{planar_rot_TO_loc_planar}
 A locally 3-connected 2-complex has a planar rotation system if and only if all contractions are 
locally planar. \footnote{\autoref{planar_rot_TO_loc_planar} will follow from 
\autoref{rot_system_exists} below.}
\end{lem}

We remark that by \autoref{sum_3con} below the class of locally 3-connected 2-complexes is 
closed 
under contractions.

\section{Vertex sums}\label{sec_vertex_sum}

In this short section we prove some elementary facts about an operation we call `vertex sum' 
which 
is used in the proof of \autoref{kura_intro}. 

Let $H_1$ and $H_2$ be two graphs with a common vertex $v$ and a bijection $\iota$ between the 
edges incident with $v$ in $H_1$ and $H_2$.
The \emph{vertex sum} of $H_1$ and $H_2$ over $v$ given $\iota$ is the graph obtained from the 
disjoint union of $H_1$ and $H_2$ by deleting $v$ in both $H_i$ and adding an edge between any pair 
$(v_1,v_2)$ of vertices $v_1\in V(H_1)$ and  $v_2\in V(H_2)$ such that $v_1v$ and $v_2v$ are 
mapped to one another by $\iota$, see \autoref{fig:vx_sum}.  
   \begin{figure} [htpb]   
\begin{center}
   	  \includegraphics[height=2.5cm]{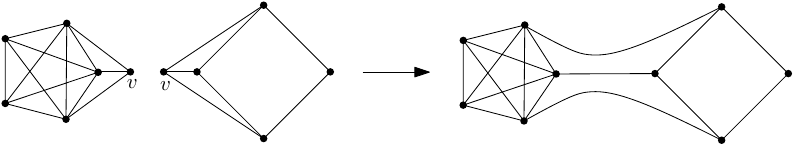}
   	  \caption{The vertex sum of the two graphs on the left is the graph on 
the right.}\label{fig:vx_sum} 
\end{center}
   \end{figure}

Let $C$ be a 2-complex with a non-loop edge $e$ with endvertices $v$ and $w$. 

\begin{obs}\label{obs1}
 The link graph of $C/e$ at 
$e$ is the vertex sum of the link graphs $L(v)$ and $L(w)$ over the common vertex $e$. 
\qed
\end{obs}

\begin{lem}\label{sum_planar1}
 Let $G$ be a graph that is a vertex sum of two graphs $H_1$ and $H_2$ over the common vertex $v$.
Let $(\sigma_x^i|x\in V(H_i))$ be a planar rotation system of $H_i$ for $i=1,2$ such that 
$\sigma_v^1$ is the inverse of $\sigma_v^2$. Then  $(\sigma_x^i|x\in V(H_i)-v, i=1,2)$ is a planar 
rotation system of $G$.
\end{lem}

\begin{proof}[Proof sketch.] This is a consequence of the topological fact that the connected sum 
of two spheres is the sphere. 
\end{proof}

\begin{lem}\label{sum_planar2}
 Let $G$ be a graph that is a vertex sum of two graphs $H_1$ and $H_2$ over the common vertex $v$.
 Assume that the vertex $v$ not a cutvertex of $H_1$ or $H_2$. 
 Assume that $G$ has a planar rotation system $\Sigma$. Then there are planar rotation systems of 
$H_1$ and $H_2$ that agree with $\Sigma$ at the vertices in $V(G)\cap V(H_i)$ and that are reverse 
at $v$.
\end{lem}

\begin{proof}
Since the vertex $v$ is not a cutvertex of the graph $H_2$, the graph $H_1$ can be obtained from 
the graph $G$ by contracting the connected vertex set $V(H_2)-v$ onto a single 
vertex.
Now let a plane embedding $\iota$ of $G$ be given that is induced by the rotation system $\Sigma$. 
Since contractions can be performed within the plane embedding $\iota$, there is a planar rotation 
system $\Sigma_1$ of the graph $H_1$ that agrees with $\Sigma$ at all 
vertices in $V(H_1)-v$.

Since the vertex $v$ is not a cutvertex of $H_1$ or $H_2$, the cut $X$ of $G$ consisting of 
the edges between $V(H_1)-v$ and $V(H_2)-v$ is actually a bond of the graph $G$. 
The bond $X$ is a circuit $o$ of the dual graph of $G$ with respect to the embedding $\iota$. And 
the rotator at $v$ of the embedding $\Sigma_1$ is equal (up to reversing) to the cyclic orientation 
of the edges on the circuit $o$.
Similarly, we construct a planar rotation system  $\Sigma_2$ of $H_2$ that agrees with $\Sigma$ at 
all 
vertices in $V(H_2)-v$, and the rotator at the vertex $v$ is the other orientation of the circuit 
$o$. This completes the proof. 
\end{proof}

\begin{proof}[Proof of \autoref{contr_pres_planar}.]
 This is a consequence of \autoref{sum_planar1} and \autoref{sum_planar2}.
\end{proof}

\begin{lem}\label{sum_3con}
 Let $G$ be a graph that is a vertex sum of two graphs $H_1$ and $H_2$ over the common vertex $v$.
 Let $k\geq 2$. 
If $H_1$ and $H_2$ are $k$-connected\footnote{Given $k\geq 2$, a graph with at least $k+1$ 
vertices is \emph{$k$-connected} if the removal of less than $k$ vertices does not make it 
disconnected. Moreover it is not allowed to have loops and if $k>2$, then it is not allowed to 
have parallel edges. }, then so is $G$. 
\end{lem}

\begin{proof} 
 Suppose for a contradiction that there is a set of less than $k$ vertices of $G$ such that $G\sm 
X$ 
is disconnected. 
Let $Y$ be the set of edges incident with $v$ (suppressing the bijection between the edges incident 
with $v$ in $H_1$ and $H_2$ in our notation). As $H_1$ is $k$-connected, the set $Y$ contains at 
least $k$ edges. If $k>2$, then since no $H_i$ has parallel edges, no two edges in $Y$ share a 
vertex. Thus in this case the set $Y$ contains $k$ edges that are vertex disjoint. If $k=2$, 
then either one $H_i$ consists of a single class of parallel edges and the 
lemma is immediate; or else, there are 
two disjoint edges of $Y$ -- here this is true as $Y$  considered as a subgraph of $G$ is a 
bipartite graph with at 
least two 
vertices on either side each having degree at least one. 

Hence by the pigeonhole principle, there is an edge $e$ in $Y$ such that no endvertex of $e$ is in 
$X$.  Let $C$ be the component of $G\sm X$ 
that contains $e$. Let $C'$ be a different component of $G\sm X$. Let $i$ be such that $H_i$ 
contains a vertex $w$ of $C'$. 

In $H_i$ this vertex $w$ and an endvertex of $e$ are separated by $X+v$. As $H_i$ is $k$-connected, 
we deduce that all vertices of $X$ are in $H_i$. Then the connected graph $H_{i+1}$ is a subset of 
$C$. Hence the vertex $w$ and an endvertex of $e$ are separated by $X$ in $H_i$. This is a 
contradiction to the assumption that $H_i$ is $k$-connected. 
\end{proof}

\vspace{.3cm}
In our proof we use the following simple fact.
\begin{lem}\label{branch_contract}
 Let $G$ be a graph with a minor $H$. Let $v$ and $w$ be vertices of $G$ contracted 
to 
the same vertex of $H$. Then there is a minor $G'$ of $G$ such that $v$ and $w$ are contracted to 
different vertices of $G'$ and their branch vertices are joined by an edge $e$ and $H=G'/e$. 
\qed
\end{lem}

\section{Constructing planar rotation systems}\label{s_constr}

The aim of this section is to prove the following lemma, which is used in the proof of 
\autoref{kura_intro}. This lemma roughly says that a 2-complex has a planar rotation system 
if and only if certain contractions are locally planar. A \emph{chord} of a cycle $o$ is an edge 
not in $o$ joining two distinct vertices in $o$ but not parallel to an edge of $o$. A cycle that 
has no 
chord is \emph{chordless}.

\begin{lem}\label{rot_system_exists}
 Let $C$ be a locally 3-connected 2-complex. 
Assume that the following 2-complexes are locally 
planar: $C$, for every non-loop edge $e$ the contraction $C/e$, and for every non-loop chordless 
cycle $o$ of 
$C$ and some $e\in o$ the contraction $C/(o-e)$.

Then $C$ has a planar rotation system. 
\end{lem}

First we show the following.

\begin{lem}\label{locally_at_edge}
Let $C$ be a 2-complex with an edge $e$ with endvertices $v$ and $w$. 
Assume that the link graphs $L(v)$ and $L(w)$ at $v$ and $w$ are 3-connected and that the link 
graph 
$L(e)$ of $C/e$ at $e$ is planar.
Then for any two planar rotation systems of $L(v)$ and $L(w)$ the rotators at $e$ are reverse of 
one 
another or agree. 
\end{lem}

\begin{proof}
Let $\Sigma=(\sigma_x|x\in (L(v)\cup L(w))-e)$ be a planar rotation system of $L(e)$. By 
\autoref{sum_planar2} there is a rotator $\tau_e$ at $e$ such that $(\sigma_x|x\in L(v)-e)$ 
together 
with $\tau_e$ is a 
planar rotation system of $L(v)$ and $(\sigma_x|x\in L(w)-e)$ together with the inverse of $\tau_e$ 
is a planar rotation system of $L(w)$.

Since $L(v)$ and $L(w)$ are 3-connected, their planar rotation system are unique up to reversing 
and hence the lemma follows.
\end{proof}

Let $C$ be a locally 3-connected 2-complex such that $C$ and for every non-loop $e$ all 
contractions $C/e$ are locally planar. We pick a planar rotation system $(\sigma_e^v| e\in 
V(L(v)))$ 
at each link graph $L(v)$ of $C$. By \autoref{locally_at_edge}, for every edge $e$ of $C$ with 
endvertices $v$ and $w$ the rotators $\sigma_e^v$ and $\sigma_e^w$ are reverse or agree. We colour 
the edge $e$ green if they are reverse and we colour it red otherwise.

A \emph{pre-rotation system} is such a choice of rotation systems such that all edges are coloured 
green. The following is an immediate consequence of the definitions.

\begin{lem}\label{pre-rot}
 $C$ has a pre-rotation system if and only if $C$ has a planar rotation system.
 \qed
\end{lem}

\begin{lem}\label{is_even}
 Let $o$ be a cycle of $C$ and $e$ an edge on $o$. Assume that the link graph $L[o,e]$ of $C/(o-e)$ 
at 
$e$ is loop-planar. Then the number of red edges of $o$ is even.
\end{lem}

\begin{proof} 
Since $L[o,e]$ is loop-planar, by \autoref{sum_planar2} there are planar rotation 
systems of all link graphs of vertices of $C$ on $o$ such that for every edge $x\in o$ with 
endvertices $v$ and $w$ the rotators $\sigma_x^v$ and $\sigma_x^w$ are reverse. 
Hence there are assignments of planar rotation systems to the link graphs at vertices of $o$ such 
the number of red edges on $o$ is zero. 

Since all link graphs are 3-connected, the planar rotation systems are unique up to reversing. 
Reversing 
a rotation system flips the colours of all incident edges. Hence for any assignment of planar 
rotation systems the number of red edges of $o$ must be even. 
\end{proof}

\begin{proof}[Proof of \autoref{rot_system_exists}.]
By \autoref{pre-rot}, it suffices to construct a pre-rotation system, that is, to construct suitable
rotation systems at each link graph of $C$.

 We may assume that $C$ is connected. We pick a spanning tree $T$ of $C$ with root $r$.
 At the link graph at $r$ we pick an arbitrary planar rotation system. Now we define a
rotation system $(\sigma_e^v| e\in V(L(v)))$ at some vertex $v$ assuming that 
for the unique neighbour $w$ of $v$ nearer to the root in $T$ we have already 
defined a rotation system $(\sigma_e^w| e\in V(L(w)))$. Let $e$ be the edge between $v$ and $w$ 
that 
is in $T$. 
By \autoref{locally_at_edge}, there is a planar rotation system $(\sigma_e^v| e\in V(L(v)))$ of the 
link graph $L(v)$ such that the rotators $\sigma_e^v$ and $\sigma_e^w$ are reverse. As $C$ is 
connected, this defines a planar rotation system at every vertex of $C$. It remains to show that 
every edge $e$ of $C$ is green with respect to that assignment. This is true by construction if $e$ 
is 
in $T$. 

\begin{lem}
 Every edge $e$ of $C$ that is not in $T$ and is not a loop is green.
\end{lem}

\begin{proof}
 Let $o_e$ be the fundamental cycle of $e$ with respect to $T$. 
We prove by induction on the number of edges of $o_e$ that $e$ is green. The base case is that 
$o_e$ is chordless. Then by assumption the link graph $L[o,e]$ of $C/(o-e)$ at $e$ is loop-planar. 
So the number of red edges on $o_e$ is even by 
\autoref{is_even}. As shown above all edges of $o_e$ 
except for possibly $e$ are green. So $e$ must be green. 

Thus we may assume that $o_e$ has chords. By shortcutting along chords we obtain a chordless 
cycle $o_e'$ containing $e$ such that each edge $x$ of $o_e'$ not in $o_e$ is a chord of $o_e$.
Thus each such edge $x$ is not in $T$ and not a loop. Since no chord $x$ can be parallel to 
$e$, the corresponding fundamental cycles $o_x$ have each strictly less edges than 
$o_e$. Hence by induction all the edges $x$ are green. Thus all edges of $o_e'$ except for possibly 
$e$ are green. Similarly as in the base case we can now apply  \autoref{is_even} to the chordless 
cycle $o_e'$ to deduce that $e$ 
is green. 
\end{proof}

\begin{sublem}
Every loop $\ell$ of $C$ is green. 
\end{sublem}
\begin{proof}
Let $v$ be the vertex incident with $\ell$. As the link graph $L(v)$ is 3-connected and 
loop-planar each of its (two) planar rotation systems must witness that $L(v)$ is loop-planar. 
Hence 
the rotation system we picked at $L(v)$ witnesses that $L(v)$ is loop planar. Thus $\ell$ is green.
\end{proof}

As all edges of $C$ are green with respect to $\Sigma$, the family $\Sigma$ is a pre-rotation 
system of $C$. Hence $C$ has a planar rotation system by \autoref{pre-rot}.
\end{proof}

\section{Marked graphs}\label{sec:marked}

In this section we prove \autoref{loc_pl_summary} and \autoref{loc_pl_summary_strict} which are 
used in the proof of 
\autoref{kura_intro}. More precisely, these lemmas characterise when a 
2-complex is 
locally planar in terms of finitely many obstructions. 

A \emph{marked graph} is a graph $G$ together with two of its vertices $v$ and $w$ and three pairs 
$((a_i,b_i)|i=1,2,3)$ of its edges, where the $a_i$ are incident with $v$ and the $b_i$ are 
incident with $w$. We stress that we allow $a_i=b_i$.

Given a 2-complex $C$, a link graph $L(x)$ of $C$, a loop $\ell$ of $C$ incident with $x$ and three 
distinct faces $f_1, f_2, f_3$ of $C$  traversing $\ell$, the marked graph \emph{associated} with 
$(x,\ell,f_1,f_2,f_3)$ is the graph $L(x)$ together with the two vertices $v$ and $w$ of $L(x)$ 
corresponding to $\ell$. The traversal of each face $f_i$ of $\ell$ corresponds to edges $a_i$ 
and $b_i$ incident with $v$ and $w$, respectively. As $f_i$ is a closed trail in $C$, each vertex 
of $L(x)$ is incident with at most one edge corresponding to $f_i$. Hence $a_i$ and $b_i$ are 
defined unambiguously. Note that if $f_i$ consists only of $\ell$, then $a_i=b_i$. 
This completes the definition of the associated marked graph $(G,v,w, ((a_i,b_i)|i=1,2,3))$. 

A marked graph $(G,v,w, ((a_i,b_i)|i=1,2,3))$ is \emph{planar} if there is a planar rotation system 
$(\sigma_x|x\in V(G))$ of $G$ such that $\sigma_v$ restricted to $(a_1,a_2,a_3)$ is the inverse 
permutation of 
$\sigma_w$ restricted to $(b_1,b_2,b_3)$ -- when concatenated with the bijective map $b_i\mapsto 
a_i$. The next lemma characterises loop-planarity. 

\begin{lem}\label{loop_to_marked}
 A 3-connected link graph $L(x)$ is loop-planar if and only if it is a planar graph and all its 
associated marked graphs are planar marked graphs. 
\end{lem}

\begin{proof}
Clearly, if $L(x)$ is loop-planar, then all its link graphs and all their associated marked 
graphs are planar. 
Conversely assume that a link graph $L(x)$ and all its associated marked graphs are planar. Then 
$L(v)$ has a planar rotation system $\Sigma$. As $L(x)$ is 3-connected, this rotation system is 
unique up to 
reversing. Hence any planar rotation system witnessing that some associated marked graph is planar 
is equal to $\Sigma$ or its inverse. By reversing that rotation system if necessary, we may assume 
that it is equal to $\Sigma$. Hence $\Sigma$ is a planar rotation system that witnesses that $L(x)$ 
is loop-planar.
\end{proof}

\begin{cor}
 A locally 3-connected 2-complex $C$ is locally planar if and only if all its link graphs and 
all their associated marked graphs are planar. 
\end{cor}

\begin{proof}
 By definition, a 2-complex is locally planar if all its link graphs are loop-planar. 
\end{proof}

A marked graph $(G,v,w, ((a_i,b_i)|i=1,2,3))$ is \emph{3-connected} if $G$ is 3-connected.
We abbreviate $A=\{a_1,a_2,a_3\}$ and $B=\{b_1,b_2,b_3\}$. 

A \emph{marked minor} of a marked graph $(G,v,w, ((a_i,b_i)|i=1,2,3))$ is obtained by doing a 
series of the following operations:
\begin{itemize}
 \item contracting or deleting an edge not in $A\cup B$;
 \item replacing an edge $a_i\in A\sm B$ and an edge $b_j\in B\sm A$ that are in parallel by a 
single new edge which is in that parallel class. In the reduced graph, this new edge is 
$a_i$ and $b_j$. 
\item the above with `serial' in place of `parallel'. 
\item apply the bijective map $(v,A)\mapsto (w,B)$.
\end{itemize}

\begin{lem}\label{minimal_minor}
 Let $\hat G=(G,v,w, ((a_i,b_i)|i=1,2,3))$ be a marked graph such that G is planar.
Let $\hat H$ be a 3-connected marked minor of $\hat G$. Then $\hat G$ is planar if and only if 
$\hat H$ is planar.
\end{lem}

Before we can prove this, we need to recall some facts about rotation systems of graphs. 
Given a graph $G$ with a rotation system $\Sigma=(\sigma_v|v\in V(G))$ and an edge $e$. The 
rotation 
system \emph{induced} by $\Sigma$ on $G-e$ is $(\sigma_v-e|v\in V(G))$. Here $\sigma_v-e$ is 
obtained from the cyclic ordering $\sigma_v$ by deleting the edge $e$. 
The rotation 
system \emph{induced} by $\Sigma$ on $G/e$ is $(\sigma_v|v\in V(G/e)-e)$ together with $\sigma_e$ 
defined as follows. Let $v$ and $w$ be the two endvertices of $e$. Then $\sigma_e$ is obtained from 
the cyclic ordering $\sigma_v$ by replacing the interval $e$ by the interval $\sigma_w-e$ (in 
such a way that the predecessor of $e$ in $\sigma_v$ is followed by the successor of $e$ in 
$\sigma_w$). Summing up, $\Sigma$ induces a rotation system at every minor of $G$. Since the class 
of plane graphs\footnote{A \emph{plane graph} is a graph together with an embedding in the plane. } 
is closed under taking minors, rotation systems induced by planar rotation systems are planar. 

\begin{proof}[Proof of \autoref{minimal_minor}]
Let $\Sigma$ be a planar rotation system of $G$. Let $\Sigma'$ be the rotation system of the 
graph $H$ of $\hat H$ induced by $\Sigma$. As mentioned above, $\Sigma'$ is planar. 

Moreover, $\Sigma$ witnesses that  $\hat G$ is a planar marked graph if and only if   $\Sigma'$ 
witnesses that  $\hat H$ is a planar marked graph. Hence if $\hat G$ is planar, so is $\hat H$. Now 
assume that $\hat H$ is planar. Since $H$ is 3-connected, it must be that $\Sigma'$ witnesses that 
the marked graph $\hat H$ is planar. Hence the marked graph $\hat G$ is planar.
\end{proof}

Our aim is to characterise when 3-connected marked graphs are planar. By \autoref{minimal_minor} 
it suffices to study that question for marked-minor minimal 3-connected marked 
graphs; we call such marked graphs  \emph{3-minimal}. 

It is reasonable to expect -- and indeed true, see below -- that there are only finitely 
many 3-minimal marked graphs. In the following we shall compute them explicitly.

Let $\hat G=(G,v,w, ((a_i,b_i)|i=1,2,3))$ be a marked graph. 
We denote by $V_A$ the set of endvertices of edges in $A$ different from $v$.
We denote by $V_B$ the set of endvertices of edges in $B$ different from $w$.

\begin{lem}\label{VA}
 Let $\hat G=(G,v,w, ((a_i,b_i)|i=1,2,3))$ be 3-minimal. Unless $G$ is $K_4$, every edge 
in $E(G)\sm (A\cup B)$ has its endvertices either both in $V_A$ or both in 
$V_B$. 
\end{lem}

\begin{proof}
By assumption $G$ is a 3-connected graph with at least five vertices such that any proper marked 
minor of 
$\hat G$ is not 3-connected. Let $e$ be an edge of $G$ that is not in $A\cup B$. By Bixby's Lemma 
{\cite[Lemma 8.7.3]{Oxley2}} either $G-e$ is 3-connected\footnote{The notion of `3-connectedness' 
used in {\cite[Lemma 8.7.3]{Oxley2}} is slightly more general than the notion used here. Indeed, 
the additional 3-connected graphs there are subgraphs of $K_3$ or subgraphs of $U_{1,3}$ -- the 
graph with two vertices and three edges in parallel. It is straightforward to check that these 
graphs 
do not come up here as they cannot be obtained from a 3-connected graph with at least 5 vertices by 
a single operation of deletion or contraction (and simplification as above). } after suppressing 
serial 
edges or $G/e$ is 3-connected after suppressing parallel edges. 

\begin{sublem}\label{del7}
 There is no 3-connected graph $H$ obtained from $G-e$ by suppressing 
serial edges.
\end{sublem}

\begin{proof}
Suppose for a contradiction that there is such a graph $H$. As $G$ is 3-connected, every class of 
serial edges of $G-e$
has size at most two.  
By minimality of $G$, there is no marked minor of $\hat G$ with graph $H$. Hence one of these 
series classes has to contain two edges in $A$ or two edges in $B$. By symmetry, we may assume that 
$e$ has an endvertex $x$ that is incident with two edges $e_1$ and $e_2$ in $A$. As $G$ is 
3-connected these two adjacent edges of $A$ can only share the vertex $v$. Thus $x=v$. This 
is a 
contradiction to the assumption that $e_1$ and $e_2$ are in series as $v$ is incident with the 
three edges of $A$.
\end{proof}

By \autoref{del7} and Bixby's Lemma, we may assume that the graph $H$ obtained from $G/e$ by 
suppressing parallel edges is 3-connected. By minimality of $G$, there is no marked minor of 
$\hat G$ with graph $H$. Hence $G/e$ has a nontrivial parallel class. And it must contain two edges 
$e_1$ and $e_2$ that are both in $A$ or both in $B$. By symmetry we may assume that $e_1$ and $e_2$ 
are in $A$. Since $G$ is 3-connected, the edges $e$, $e_1$ 
and $e_2$ form a triangle in $G$. The common vertex of $e_1$ and $e_2$ is $v$. Thus both 
endvertices of $e$ are in $V_A$.
\end{proof}

A consequence of \autoref{VA} is that every 3-minimal marked graph has at most most 12 edges. 
However, we can say more:

\begin{cor}\label{cor100}
 
Let $\hat G=(G,v,w, ((a_i,b_i)|i=1,2,3))$ be 3-minimal. Then $G$ has at most 
five 
vertices.
\end{cor}

\begin{proof}
 Let $G_A$ be the induced subgraph with vertex set $V_A+v$. Let $G_B$ be the induced subgraph 
with vertex set $V_B+w$.
Note that $G=G_A\cup G_B$. 
If $G_A$ and $G_B$ have at least three vertices in common, then $G$ has at most five vertices as 
$G_A$ and $G_B$ both have at most four vertices. Hence we may assume that $G_A$ and $G_B$ have at 
most two vertices in common. As $G$ is 3-connected, the set of common vertices cannot be a 
separator of $G$. Hence $G_A\se G_B$ or $G_B\se G_A$. Hence $G$ has at most four vertices in this 
case. 
\end{proof}

An \emph{unlabelled marked graph} is a graph $G$ together with vertices $v$ and $w$ and edge sets 
$A$ and $B$ of size three such that all edges of $A$ are incident with $v$ and all edges in $B$ are 
incident with $w$. The \emph{underlying} unlabelled marked graph of a marked graph $(G,v,w, 
((a_i,b_i)|i=1,2,3))$ is $G$ together with $v$, $w$ and the sets $A=\{a_1,a_2,a_3\}$ and 
$B=\{b_1,b_2,b_3\}$.  
Informally, an unlabelled marked graph is a marked graph without the bijection between the sets 
$A$ 
and $B$. For a planar 3-connected unlabelled marked graph, there are three bijections between $A$ 
and $B$ for which the associated marked graph is planar as a marked graph. For the other three 
bijections it is not planar. 

   \begin{figure} [htpb]   
\begin{center}
   	  \includegraphics[height=2.5cm]{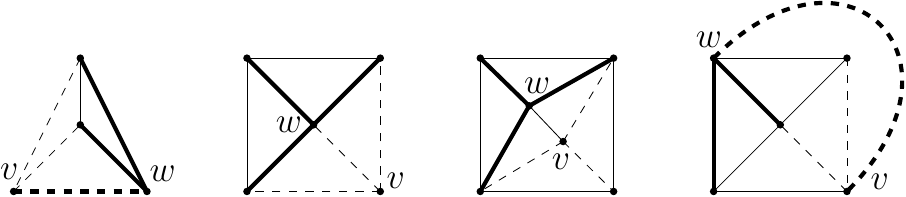}
   	  \caption{The four unlabelled marked graphs in $\Xcal$. The edges in 
$A$ are depicted dotted, the ones in $B$ are bold.}\label{fig:minimal} 
\end{center}
   \end{figure}

Marked graphs $\hat G=(G,v,w,((a_i,b_i)|i=1,2,3))$ associated to link graphs always have the 
property that the vertices $v$ and $w$ are distinct. 3-minimal marked graphs need not have this 
property. Of particular interest to us is the class $\Xcal$ depicted in \autoref{fig:minimal}; 
indeed, they describe the 3-connected marked graphs with the property that $v\neq 
w$ that are marked minor minimal with $G$ planar, as shown in the following. We shall refer to the 
four members of $\Xcal$ in the linear ordering given by accessing \autoref{fig:minimal} from left 
to right (and say things like `the first member of $\Xcal$').

\begin{lem}\label{char_real}
Let $\hat G=(G,v,w,((a_i,b_i)|i=1,2,3))$ be a 3-connected marked graph  with $v\neq w$ and $G$ 
planar. Then $\hat G$ has a marked minor that has an underlying 
unlabelled marked 
graph in $\Xcal$.
\end{lem}

\begin{proof}
By \autoref{cor100}, $\hat G$ has a marked minor minimal 3-connected marked minor $\hat 
H=(H,v,w,((a_i,b_i)|i=1,2,3))$, where $H$ has at most five vertices. 

\begin{sublem}\label{cases}
 The only 3-connected planar graphs with at most five vertices are $K_4$, the 4-wheel and $K_5^-$.
\end{sublem}

\begin{proof}
Since $K_4$ is the only 3-connected graph with less than five vertices, it suffices to consider the 
case where the graph $K$ in question has five vertices. As five is an odd number and $K$ has 
minimum degree 3, $K$ has a vertex $v$ of degree 4. Hence $K-v$ is 2-connected. Hence it has to 
contain a 4-cycle. Thus $K$ has the 4-wheel as a subgraph. Thus $K$ is  the 4-wheel, $K_5^-$ or 
$K_5$. As $K$ is planar, it cannot be $K_5$.
\end{proof}

By \autoref{cases}, $H$ is $K_4$, the 4-wheel or $K_5^-$. In the following we treat 
these cases separately.
As above we let  
$A=\{a_1,a_2,a_3\}$ and $B=\{b_1,b_2,b_3\}$.
 
{\bf Case 1:} $H=K_4$.  If the vertices $v$ and $w$ of $H$ are distinct, then  the underlying 
unlabelled marked graph of $\hat{H}$ is the first member of $\Xcal$ and the lemma is true in this 
case.
Suppose for a contradiction that $v=w$. Then each edge incident with $v$ is in $A$ and $B$. Let 
$H'$ be the marked graph obtained from $\hat H$ by replacing each edge incident with $v$ by two 
edges in parallel, one in $A$, one in $B$. It is clear that $H'$ is a marked minor of $\hat G$. 
By applying \autoref{branch_contract} to the graph of $H'$, we deduce that $G$ has $K_5$ as a minor.
This is a contradiction to the assumption that $G$ is planar.

{\bf Case 2A:} $H$ is the 4-wheel and $v\neq w$.

{\bf Subcase 2A1:} $v$ or $w$ is the center of the 4-wheel.
By applying the bijective map $(v,A)\mapsto (w,B)$ if necessary, we may assume that $w$ is the 
center. 
Our aim is to show that the underlying 
unlabelled marked graph of $\hat H$ is the second member of $\Xcal$.
As $v$ has degree three, $A$ is as desired. 
By \autoref{VA}, the two edges on the rim not in $A$ must have both their endvertices in $V_B$. 
Hence $B$ is as desired. Thus the underlying 
unlabelled marked graph of $\hat H$ is the second member of $\Xcal$.

{\bf Subcase 2A2:}  $v$ and $w$ are adjacent vertices on the 
rim. We shall show that this case is not possible. Suppose for a contradiction that it is possible.

We denote by $e$ the edge on the rim not incident with $v$ or $w$. One endvertex has distance two 
from $v$, the other has distance two from $w$. Hence the endvertices of $e$ cannot both be in $V_A$ 
or both be in $V_B$. This is a contradiction to \autoref{VA}. 

{\bf Subcase 2A3:} $v$ and $w$ are opposite vertices 
on the rim. We shall show that this case is not possible. Suppose for a contradiction that it is 
possible.

There is an edge incident with the center not incident with $v$ or $w$. Deleting 
that edge and suppressing the vertex of degree two gives a marked 
graph whose graph is $K_4$. Hence $\hat H$ is not minimal in that case, a contradiction.
This completes Case 2A. 

 {\bf Case 2B:} $H$ is the 4-wheel and $v=w$. 
By \autoref{VA}, every edge not in $A\cup B$ must have both endvertices in $V_A$ or $V_B$. Hence 
$v$ can only be the center of the 4-wheel. 
By the minimality of $\hat H$ and by \autoref{VA}, each edge of 
the rim has both its endvertices in $V_A$ or in $V_B$. At most two edges of the rim can have all 
their endvertices in $V_A$ and in that case these edges are adjacent on the rim. The same is 
true for $V_B$. 

We denote the vertices of the rim by $(v_i|i\in \Zbb_4)$, where $v_iv_{i+1}$ is an edge. By 
symmetry, we may assume that $v_1$ is the unique vertex of the rim not in $V_A$. Then $v_3$ must be 
the unique vertex of the rim not in $V_B$. It follows that the edges $vv_2$ and $vv_4$ are in $A$ 
and $B$. Let $H'$ be the marked graph obtained from $\hat H$ by replacing each of $vv_2$ and 
$vv_4$ by two edges in parallel, one in $A$, one in $B$. It is clear that $H'$ is a marked minor of 
$\hat G$. Let $H''$ be the marked graph obtained from $H'$ by applying \autoref{branch_contract}. 
The underlying unlabelled marked graph of $H''$ is the third member of $\Xcal$.

  {\bf Case 3:} $H$ is $K_5^-$. 
 
 We shall show that the underlying unlabelled marked graph  of $\hat H$ is the forth graph of 
$\Xcal$. 
 $H$ has three vertices of degree four, which lie one a common 3-cycle. Removing any edge of that 
3-cycle gives a graph isomorphic to the 4-wheel. Hence by minimality of $\hat H$, it must be that 
this 3-cycle is a subset of $A\cup B$. In particular, $v$ and $w$ are distinct vertices on that 
3-cycle. Up to symmetry, there is only one choice for 
$v$ and $w$. 
By applying the map $(v,A)\mapsto (w,B)$ if necessary, we may assume that $A$ contains at least 
two edges of that 3-cycle. 

We denote the two vertices of $H$ of degree three by $u_1$ and $u_2$. We denote the vertex of 
degree four different from $v$ and $w$ by $x$. 
By exchanging the roles of $u_1$ and $u_2$ if necessary, we may assume that $A=\{vw, vx, 
vu_1\}$. 

Recall that $wx\in B$. 
The endvertex $u_2$ of the edge $vu_2$ is not in $V_A$ and this edge cannot be in $B$. Hence by 
\autoref{VA}, both its endvertices must be in $V_B$. Hence $vw\in B$ and $wu_2\in 
B$. 
Summing up 
$B=\{wx, vw, wu_2\}$. Thus in this case the underlying unlabelled graph of $\hat H$ is the 
forth graph of $\Xcal$. 
\end{proof}

By $\Ycal$ we denote the class of marked graphs that are not planar as marked graphs and whose 
underlying unlabelled marked graphs are isomorphic to a member of $\Xcal$ -- perhaps after applying 
the bijective map $(v,A)\mapsto (w,B)$. We consider two marked graphs the same if they have the 
have the same graph and the 
same bijection between the sets $A$ and $B$ (although the elements in $A$ might have different 
labels). Hence for each $X\in \Xcal$, there are precisely three 
marked 
graphs in $\Ycal$ with underlying unlabelled marked graph $X$, one for each of the three bijections 
between $A$ and $B$ that are not compatible with any rotation system of the graph of $X$ (which is 
3-connected). Thus $\Ycal$ has twelve elements. 

Summing up we have proved the following. 

\begin{lem}\label{loc_pl_summary}
 A locally 3-connected 2-complex is locally planar if and only if all its link graphs are planar 
and 
all their associated marked graphs do not have a marked minor from $\Ycal$. 
\end{lem}

\begin{proof}
Since no marked graph in $\Ycal$ is planar, it is immediate that if a  2-complex is locally planar, 
then all its link graphs are planar 
and all their associated marked graphs do not have a marked minor from $\Ycal$. 

For the other implication it suffices to show that any 3-connected link graph $L(x)$ that is planar 
but not loop-planar has an associated marked graph that has a marked minor in $\Ycal$. 
By \autoref{loop_to_marked}, $L(x)$ has an associated marked graph $\hat G$ that is not planar. By 
\autoref{char_real}, $\hat G$ has a marked minor $\hat H$ whose underlying unlabelled marked graph 
is in 
$\Xcal$. By \autoref{minimal_minor}, $\hat H$ is not planar. Hence $\hat H$ is in $\Ycal$. 
\end{proof}

\autoref{loc_pl_summary} has already the following consequence, which characterises 
embeddability in 3-space by finitely many obstructions.\footnote{As turns out, \autoref{weak_kura} 
is too weak to be used directly in our proof of \autoref{kura_intro}. Indeed, in our proof it will 
not always be possible to contract $C$ onto a single vertex but we need to choose the edges we 
contract carefully (using the additional information provided in \autoref{rot_system_exists}). } 

\begin{cor}\label{weak_kura}
 Let $C$ be a simply connected locally 3-connected 2-complex. Let $C'$ be a contraction of $C$ to a 
single vertex $v$. 
 Then $C$ has an embedding into \Sthree\ if and only if no marked graph associated to the link 
graph at $v$ has a marked minor in the finite set $\Ycal$. 
\end{cor}

\begin{proof}
By \autoref{emb_to_rot}, $C$ is embeddable if and only if it has a 
planar rotation system. By \autoref{planar_rot_TO_loc_planar} $C$ has a planar rotation system if 
and only if $C'$ is locally planar. Hence \autoref{weak_kura} follows from \autoref{loc_pl_summary}.
\end{proof}

In the following we will deduce from \autoref{loc_pl_summary} a more technical analogue. 
A \emph{strict marked graph} is a marked graph $(G,v,w, ((a_i,b_i)|i=1,2,3))$ together with a 
bijective map between the edges incident with $v$ and the edges incident with $w$ that maps $a_i$ 
to 
$b_i$. A \emph{strict marked minor} is obtained by deleting edges not incident with 
$v$ or $w$ or deleting an edge not in $A\cup B$ incident with $v$ and the edge it is bijected to,
and contracting edges if they have an endvertex $x$ of degree two such that 
$x$ is neither equal to $v$ or $w$ nor $x$ is adjacent to $v$ or $w$. 
We also allow to apply the bijective map $(v,A)\mapsto (w,B)$. 
\begin{rem}
 We call this relation the `strict marked minor relation' as it is more restrictive than the 
`marked minor relation'.
\end{rem}

The proof of the next lemma is technical. We invite the reader to skip it when first reading the 
paper. 

\begin{lem}\label{equal}
There is a finite set $\Ycal'$ of strict marked graphs such that a strict marked graph has a strict 
marked minor in $\Ycal'$ if and only if its marked graph has a 
marked minor in $\Ycal$.
\end{lem}

\begin{proof}The \emph{underlyer} of a strict marked graph $\hat Y$ is the the underlying 
unlabelled marked 
graph of the strict marked graph $\hat Y$.
We define $\Ycal'$ and reveal the precise definition in steps during the proof. Now we reveal that 
by $\Ycal'$ we denote the class of strict marked graphs with underlyer in $\Xcal_5$ -- perhaps 
after 
applying 
the bijective map $(v,A)\mapsto (w,B)$. The set $\Xcal_5$, however, is revealed later. We 
abbreviate `strict marked minor' by \emph{5-minor}. 
We define \emph{0-minors} like `marked minors' but on the larger class of strict marked graphs 
where we additionally allow that edges incident with $v$ or $w$ have no image under $\iota$. (This 
is necessary for this class to be closed under 0-minors). Let $\Xcal_0=\Xcal$.

Let $\hat Y$ be a strict marked graph. In this language, it suffices to show that $\hat Y$ has a 
0-minor with underlyer in $\Xcal_0$ if and only if $\hat Y$ has a 5-minor with underlyer in 
$\Xcal_5$.
We will show this in five steps. In the $n$-th step we define \emph{n-minors} and a set 
$\Xcal_n$ of unlabelled marked graphs and prove that $\hat Y$ has an $(n-1)$-minor with underlyer 
in $\Xcal_{n-1}$ if and only if  $\hat Y$ has an $n$-minor with underlyer in $\Xcal_{n}$. 

Starting with the first step, we define \emph{1-minors} like `0-minors' where we do not allow to 
contract edges incident with $v$ or $w$. We define $\Xcal_1$ and reveal it during the proof of the 
following fact. 
\begin{sublem}\label{step1}
 $\hat Y$ has a $0$-minor with underlyer 
in $\Xcal_0$ if and only if  $\hat Y$ has a $1$-minor with underlyer in $\Xcal_{1}$.
\end{sublem}

\begin{proof}

Assume that $\hat Y$ has a 0-minor $\hat Y_0$ with underlyer in $\Xcal_0$. 
So there is a 1-minor $\hat Y_1$ of $\hat Y$ so that we obtain $\hat Y_0$ from $\hat Y_1$ by 
contracting edges incident with $v$ or $w$. We reveal that $\Xcal_1$ is a superset of $\Xcal_0$. 
Hence we may assume that there is an edge of $\hat Y_1$ that is not in $\hat Y_0$. By symmetry, we 
may assume that it is incident with $v$. We denote that edge by $e_v$, see \autoref{fig:sit}. 

   \begin{figure} [htpb]   
\begin{center}
   	  \includegraphics[height=1.25cm]{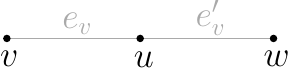}
   	  \caption{The situation of the proof of \autoref{step1}.}\label{fig:sit} 
\end{center}
   \end{figure}

We may assume that $\hat Y_1$ is minimal, that is, it has no proper 1-minor that has a 0-minor 
isomorphic to $\hat Y_0$. Applying this to $\hat Y_1-e_v$, yields that there must be an 
edge $e_v'$ incident with $v$ in $\hat Y_0$ that in $\hat Y_1$ is not incident with 
$v$ but the other endvertex of $e_v$. In particular, the edge $e_v'$ is not in $A$. Let $u$ be the 
common vertex of $e_v$ and $e_v'$. 

Next we show that $u$ is only incident with $e_v$ and $e_v'$ in $Y_1$.
By going through the four unlabelled marked graphs in 
$\Xcal_0=\Xcal$, we check that there is at most one edge incident with $v$ but not in $A$. Hence 
$u$ can only be incident with edges not in $\hat Y_0-e_v'$. Moreover the connected component of 
$Y_1\sm Y_0$ containing $u$ can only contain $v$ and vertices not incident with any edge of $Y_0$. 
Thus by the minimality of $\hat Y_1$, this connected component only contains the edge $e_v$. So 
$u$ is 
only incident with $e_v$ and $e_v'$.

Since $u$ has degree 2, $\hat Y_1/e_v'$ has a 0-minor isomorphic to $\hat Y_0$. By the minimality 
of $\hat Y_1$, it must be that  $\hat Y_1/e_v'$ is not 1-minor of it. Hence $e_v'$ has to be 
incident with $w$. 

Suppose for a contradiction that there is an edge $e_v$ and an edge $e_w$ defined as $e_v$ with 
`$w$' in place of `$v$'. Then as each member of $\Xcal$ has at most one edge between $v$ and $w$, 
it must be that $e_v'=e_w'$. This is a contradiction as $e_v'$ is incident with $w$ but not with 
$v$ in $\hat Y_1$ and for $e_w'$ it is the other way round.

Summing up, we have shown that $\hat Y_1$ is either equal to $\hat Y_0$ or otherwise $\hat Y_0$ has 
an 
edge $e$ between $v$ and $w$ and $\hat Y_1$ is obtained by subdividing that edge. This edge $e$ 
cannot be in $A\cap B$. 

Now we reveal that we define $\Xcal_1$ from $\Xcal$ by adding two more unlabelled marked graphs as 
follows, see 
\autoref{fig:minimal_strict}.
   \begin{figure} [htpb]   
\begin{center}
   	  \includegraphics[height=3.5cm]{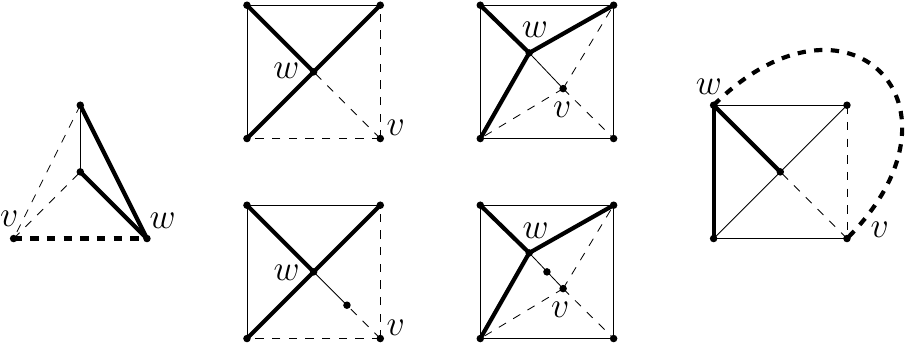}
   	  \caption{The six unlabelled marked graphs in $\Xcal_1$. The edges in 
$A$ are depicted dotted, the ones in $B$ are bold.}\label{fig:minimal_strict} 
\end{center}
   \end{figure}
The first we get from the second member by subdividing the edge between $v$ and $w$ and let the 
subdivision edge incident with $v$ remain in $A$. The second we get from the third member by 
subdividing the edge between $v$ and $w$.

From this construction it follows that if $\hat Y$ has a 0-minor $\hat Y_0$ with underlyer 
in $\Xcal_0$, then the 1-minor $\hat Y_1$ of $\hat Y$ defined above has an underlyer in $\Xcal_1$. 
Hence $\hat Y$ has a 1-minor with underlyer in $\Xcal_1$ if and only if it has a 0-minor with 
underlyer in $\Xcal_0$. 
\end{proof}

 Starting with the second step, we define \emph{2-minors} like `1-minors' where we only allow to 
delete edges incident with $v$ and $w$ in the pairs given by the bijection $\iota$ -- and if they 
are not in $A\cup B$.  
We obtain $\Xcal_2$ from $\Xcal_1$ by adding the following unlabelled marked graphs. 
For each member of  $\Xcal_1$ such that all edges incident with $v$ or $w$ are in $A\cup B$ we add 
no new member. There is one member in $X\in \Xcal_1$ that has an edge incident with $w$ not in 
$A\cup B$ but every edge incident with $v$ is in $A$. We add new members obtained from $X$ by 
adding one more edge incident to $v$ and one other vertex; this may be a vertex of $X-v$ or a new 
vertex.   
All other members of $X'\in \Xcal_1$ have the property that they have exactly one edge incident 
with $v$  not in $A\cup B$ and exactly one edge incident with $w$ not in $A\cup B$. We add new 
members to $\Xcal_2$ obtained from such an $X'$ by 
adding two more non-loop edges, one incident with $v$, the other incident with $w$.\footnote{There 
are some 
technical conditions we could further force these newly added edges to satisfy. For example, 
there are ways in which we could add two edges to the forth member of $\Xcal$ such that the 
resulting unlabelled marked graph has another member of $\Xcal$ as a strict marked minor. 
This would give 
rise to a slightly stronger version of \autoref{equal} and thus of \autoref{kura_intro}. To 
simplify the presentation we do not do it here. }  This completes the definition of $\Xcal_2$. 

\begin{sublem}\label{step2}
 $\hat Y$ has a $1$-minor with underlyer 
in $\Xcal_1$ if and only if  $\hat Y$ has a $2$-minor with underlyer in $\Xcal_{2}$.

\end{sublem}
\begin{proof}
By construction, if  $\hat Y$ has a 2-minor with underlyer in $\Xcal_{2}$, then it has a $1$-minor 
with underlyer 
in $\Xcal_1$. Now conversely assume that  $\hat Y$ has a $1$-minor $\hat Y_1$ with 
underlyer 
in $\Xcal_1$. We define $\hat Y_2$ like  `$\hat Y_1$' except that we only delete edges incident 
with $v$ or $w$ if also their image under $\iota$ is deleted. It remains to show that the 
underlyer of $\hat Y_2$ is in $\Xcal_2$, that is, the graph $Y_2$ has no loops. This is true as the 
graph $Y_1$ has no loops and the additional edges of $Y_2$ are incident with $v$ or $w$. So they 
cannot be loops as no edge of $\hat Y$ incident with $v$ or $w$ is contracted by the definition of 
1-minor.
\end{proof}
 
 Starting with the third step, we define \emph{3-minors} like `2-minors' where we do not 
allow to replace parallel or serial pairs of edges in $A\cup B$ as in the second and third 
operation of marked minor. Each member of $\Xcal_2$ has at most one edge in $A\cap B$. 
We obtain $\Xcal_3$ from $\Xcal_2$ by adding two new member for each $X\in \Xcal_2$ that has an 
edge $e$ in $A\cap B$. The first one we obtain by replacing the edge $e$ by two edges in 
parallel, one in $A\sm B$ and the other in 
$B\sm A$. The second member we construct the same with `parallel' replaced by `serial'. The 
following is immediate. 
\begin{sublem}\label{step3}
 $\hat Y$ has a $2$-minor with underlyer 
in $\Xcal_2$ if and only if  $\hat Y$ has a $3$-minor with underlyer in $\Xcal_{3}$.\qed
\end{sublem}

We define \emph{4-minors} like `3-minors' where we only allow to contract edges $e$ if they have 
an endvertex $x$ of degree two (and as before $e$ is not incident with $v$ or $w$).
\begin{dfn}
 We say that a graph $H$ is obtained from a graph $G$ by \emph{coadding} the edge $e$ of $H$ at the 
vertex $z$ of $G$ if $H/e=G$, and the edge $e$ is contracted onto the vertex $z$ of $G$, and $e$ 
is not a loop in $H$. 
\end{dfn}
We obtain $\Xcal_3$ from $\Xcal_4$ by adding all marked graphs obtained from marked graphs in 
$\Xcal_3$ 
by coadding edges $e$ at vertices different from $v$ and $w$ such that both endvertices of $e$ have 
degree at least three. 
We remark that $\Xcal_4$ is finite as any coadding of such an edge strictly reduces the 
degree-sequence of the graph in the lexicographical order. 
\begin{sublem}\label{step4}
 $\hat Y$ has a $3$-minor with underlyer 
in $\Xcal_3$ if and only if  $\hat Y$ has a $4$-minor with underlyer in $\Xcal_{4}$.
\end{sublem}
\begin{proof}
Clearly every marked graph in $\Xcal_4$ has a 3-minor in $\Xcal_3$. 
Now assume that $\hat Y$ has a 3-minor $\hat H$ with underlyer in $\Xcal_3$. We do the minors as 
before but only contract edges contracted before if they have an endvertex of degree two; and if 
they have an endvertex of degree one or are loops, we delete them instead. The resulting strict 
marked graph $\hat G$ has $\hat H$ as a 3-minor; namely we just need to contract the edges in 
$E(G)\sm E(H)$. However, both endvertices of these edges have degree at least three and they are 
not loops; that is, $G$ can be obtained from $H$ by coadding edges. Thus $\hat G$ is in $\Xcal_4$. 
So $\hat Y$ has a 4-minor in $\Xcal_4$.
\end{proof}

We define \emph{5-minors} like `4-minors', where we additionally require that the endvertex $x$ 
of degree two is not adjacent to $v$ or $w$.
We let $\Xcal_5$ to consists of those marked graphs obtained from a marked graph of 
$\Xcal_4$ by subdividing each edge incident with $v$ or $w$ at most once. 
\begin{sublem}\label{step5}
 $\hat Y$ has a $4$-minor with underlyer 
in $\Xcal_4$ if and only if  $\hat Y$ has a $5$-minor with underlyer in $\Xcal_{5}$.
\end{sublem}
\begin{proof}
Clearly every marked graph in $\Xcal_5$ has a 4-minor in $\Xcal_4$. 
5-minors are slightly more restricted than 4-minors in that there are a few edges we are not 
allowed to contract. These edges have an endvertex of degree two that is adjacent to $v$ or $w$. 
Hence if $\hat Y$ has a $4$-minor with underlyer in $\Xcal_4$, and we do the minors as before but 
do not contract the edges forbidden for 5-minors, we get a strict marked graph with underlyer in  
$\Xcal_5$, which then is a 5-minor of $\hat Y$.
\end{proof}

It is clear from the definitions that 5-minors are just strict marked minors. 
By \autoref{step1}, \autoref{step2}, \autoref{step3}, \autoref{step4} and 
\autoref{step5}, any strict marked graph has a strict 
marked minor with underlyer in $\Xcal_5$ if and only if its marked graph has a 
marked minor with underlyer in $\Xcal_0$. This completes the proof. 
\end{proof}

The set $\Ycal'$ is defined explicitly in the proof of \autoref{equal}. We fix the set $\Ycal'$ as 
defined in that proof. 
The following is analogue to \autoref{loc_pl_summary} for 
strict marked minors. 

\begin{lem}\label{loc_pl_summary_strict}
 A locally 3-connected 2-complex is locally planar if and only if all its link graphs are planar 
and all their associated strict marked graphs do not have a strict marked minor from $\Ycal'$. 
\end{lem}
\begin{proof}
 This is a direct consequence of \autoref{loc_pl_summary} and \autoref{equal}.
\end{proof}

\section{Space minors}\label{sec:space}

In this sections we introduce `space minors' and prove \autoref{kura_intro} and 
\autoref{kura_intro_hom}. 

\subsection{Motivation}
Our approach towards Lov\'asz question mentioned in the Introduction is based on 
the 
following two lines of thought. 

The first line is as follows. Suppose that a 2-complex $C$ can be embedded in \Sthree\, then we can 
define a dual graph $G$ of the embedding as follows. Its vertices are the components of $\Sbb^3\sm 
C$ and its edges are the faces of $C$; each edge is incident with the two components of $\Sbb^3\sm 
C$ touched by its face. It would be nice if the minor operations on the dual graph would 
correspond 
to minor operations on $C$. 

The operation of contraction of edges of $G$ corresponds to 
deletion of faces. But which operation corresponds to deletion of edges of $G$? If the face of $C$ 
corresponding to the edge of $G$ is incident with at most two edges of $C$, then this is the 
operation of contraction of faces (that is, identify the two incident edges along the face). For 
faces of size three, however, it 
is less clear how such an operation could be defined.  

The second line of thought is that we would like to define the minor operation such that we can 
prove an analogue of Kuratowski's theorem -- at least in the simply connected case. 

\autoref{weak_kura} above is already a characterisation of embeddability in 3-space by finitely 
many 
obstructions. However, the reduction operations are not directly operations on 2-complexes (some 
are 
just defined on their link complexes). But does \autoref{weak_kura} imply such a Kuratowski 
theorem?
Thus our aim is to define three operations on 2-complexes that correspond to
\begin{enumerate}
\item contraction of edges that are not loops\footnote{Contractions of loops do not preserve 
embeddability in general (as $\Sbb^3/\Sbb^1\not\cong \Sbb^3$).};
\item deletion of edges in link graphs;
\item contraction of edges in link graphs.
\end{enumerate}

So we make our first operation to be just the first one: contraction of edges that are not loops. A 
natural choice for the second operation is deletion of faces. This very often corresponds to 
deletion 
of edges in the link graph. In some cases however it may happen that a face corresponds to more 
than one edge in a link graph. This is a technicality we will consider later.
Also note that contraction of edges and deletion of faces are "dual"; that is, given a 2-complex 
$C$ embedded in 3-space and the dual complex $D$ (this is the dual graph $G$ defined above with 
a face attached for every edge $e$ of $C$ to the edges of $G$ incident with $e$), contracting 
an edge in $C$ results in deleting a face 
in $D$, and vice versa. This is analogous to the fact that deleting an edge in a plane 
graph corresponds to contracting that edge in the plane dual. 

For the third operation we have some freedom. One operation that corresponds to 3 is the inverse 
operation of contracting an edge. However this would not be compatible with the first 
line of thought and we are indeed able to make such a compatible choice as follows.

If an edge of the link graph corresponds to a face of $C$ that is incident with only two edges of 
$C$, then contracting that face corresponds to contracting the corresponding edge in the link 
graph. 
It is not clear, however, how that definition could be extended to faces of size three (in 
particular if all edges incident with that face are loops; which we have to deal with as we allow 
contractions of edges of $C$). 

Our solution is the following. Essentially, we are able to show that in order to construct a 
bounded 
obstruction in any non-embeddable 2-dimensional simplicial complex (which 
is the crucial step in 
a proof of a Kuratowski type theorem) that is nice enough, we only need to contract faces incident 
with two 
edges but not those of size three! Here `nice enough' means simply connected and locally 
3-connected. Both these conditions can be interpreted as face maximality conditions on the complex, 
see {\cite[\autoref*{general}]{{3space2}}}. 
`Essentially' here means that additionally we have to allow for the following two (rather
simple) operations.

If the link graph at a vertex $v$ of a 2-complex $C$ is disconnected, the 
2-complex obtained from 
$C$ by \emph{splitting} the vertex $v$ is obtained by replacing $v$ by one new vertex for each 
connected component $K$ of the link graph that is incident with the edges and faces in $K$.

Given an edge $e$ in a 2-complex $C$, the 2-complex obtained from $C$ by \emph{deleting the edge 
$e$} is obtained from $C$ by replacing $e$ by parallel edges such that each new edge 
is incident with precisely one face (for an example, see the deletion of the edge $g$ in 
\autoref{fig:space_minor2}). 

\begin{rem}(On a variant of space minors and \autoref{kura_intro}). 
 In our proof we only ever split vertices directly after deleting edges or faces, and after such a 
deletion we can without changing the proof always split the incident vertices. Hence we could 
modify these two operations so that we always afterwards additionally split all vertices incident 
with the deleted edge or face. This way we would only have four space minor operations, one for each 
corner of \autoref{4ops}. And \autoref{kura_intro} would be true in this form. 
\end{rem}

Formally, let $f$ be a face of size two in a 2-complex $C$, the 2-complex $C/f$ obtained from $C$ 
by \emph{contracting} the face $f$ is obtained from $C$ by replacing the face $f$ and its two 
incident edges by a single edge (also denoted by $f$). This new edge is incident with all faces 
that are incident with one of the two edges of $f$ -- and it is incident with the same vertices as 
$f$.

\subsection{Basic properties}

A \emph{space minor} of a 2-complex is obtained by successively performing one of the five 
operations.
\begin{enumerate}
 \item contracting an edge that is not a loop;
 \item deleting a face (and all edges or vertices only incident with that face);
 \item contracting a face of size one\footnote{Although we do not need it in our proofs, it seems 
natural to allow contractions of faces of size one. } or two if its two edges are not loops;
 \item splitting a vertex;
 \item deleting an edge (we also refer to that operation as `forgetting the incidences of an edge').
\end{enumerate}

\begin{rem}
 A little care is needed with contractions of faces. This can create faces traversing edges 
multiple times. In this paper, however, we do not contract faces consisting of two 
loops and we only perform these operations on 2-complexes whose faces have size at most three. 
Hence it could only happen that after contraction some face traverses an edge twice but in 
opposite direction. Since faces have size at most three, these traversals are adjacent. In this 
case 
we omit the two opposite traversals of the edge from the face. We delete faces incident with no 
edge. This ensures that the class of 2-complexes with faces of size at most three is closed under 
face contractions. 
\end{rem}

A 2-complex is \emph{3-bounded} if all its faces are incident with at most three edges. 
The closure 
of the class of simplicial complexes by space minors is the class of 3-bounded 2-complexes.  

It is easy to see that the space minor operations preserve 
embeddability in \Sthree\ (or in any 
other 3-dimensional manifold) and the first three commute when defined.\footnote{In order for the 
contraction of a face to be defined we need the face to have at 
most two edges. This may force contractions of edges to happen before the contraction of the 
face.}
\begin{lem}\label{well-founded}
 The space minor relation is well-founded.
\end{lem}
\begin{proof}
The \emph{face degree} of an edge $e$ is the number of faces incident with $e$. 
We consider the sum $S$ of all face degrees ranging over all edges. 
 None of the five above operations increases $S$. And 1, 2 and 3 
always strictly decrease $S$. Hence we can apply 1, 2 or 3 only a bounded number of times. 

Since no operation increases the sizes of the faces, the 
total number of vertices and edges incident with faces is bounded. Operation 4 increases the number 
of vertices and preserves the number of edges. For operation 5 it is the other way round. Hence we 
can also only apply\footnote{We exclude applications of 4 to a vertex whose link graph is 
connected and applications of 5 to edges incident with a single face. } 4 and 5 a bounded number of 
times. 
\end{proof}

\begin{lem}\label{rot_closed_down}
If a 2-complex $C$ has a planar rotation system, then all its space minors do. 
\end{lem}

\begin{proof}
By \autoref{contr_pres_planar} existence of planar rotation systems is preserved by contracting 
edges that are not loops.  Clearly the operations 2, 4 and 5 preserve planar rotation systems as 
well. Since contracting a face of size two corresponds to locally in the link graph contracting the 
corresponding edges, contracting faces of size two preserves planar rotation systems as noted after 
\autoref{minimal_minor}. The operation that corresponds to contracting a face of size one in the 
link graph is 
explained in \autoref{fig:contr_e}. 
  \begin{figure} [htpb]   
\begin{center}
   	  \includegraphics[height=2cm]{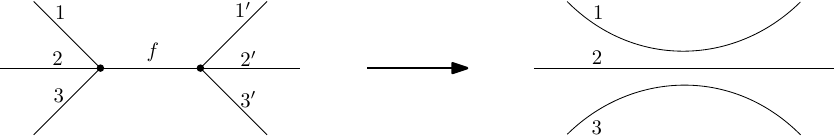}
   	  \caption{The operation that in the link graph corresponds to contracting a face $f$ 
only incident with a single edge $\ell$. The edge $\ell$ must be a loop. Hence in the link graph we 
have two vertices for $\ell$ which are joined by the edge $f$. On the left we depicted that 
configuration. Contracting $f$ in the complex yields the configuration on the right. Formally, we 
delete $f$ and both its endvertices and add for each face $x$ of size at least 
two  traversing $\ell$ an edge as follows. Before the contraction, the link graph contains two 
edges corresponding to  the traversal of $x$ of $\ell$. These edges have precisely two distinct 
endvertices that are not vertices corresponding to $\ell$. We add an edge between these two 
vertices.
}\label{fig:contr_e} 
\end{center}
   \end{figure}
It clearly preserves embeddings in the plane. Thus contracting 
a face of size one also preserves planar rotation systems.
\end{proof}

{\subsection{Generalised Cones}\label{sec:gen_cone}}

In this subsection we define the list $\Zcal$ of obstructions appearing in \autoref{kura_intro} and 
prove basic properties of the related constructions. 

Given a graph $G$ without loops and a partition $P$ of its vertex set into connected sets, the 
\emph{generalised 
cone} over $G$ with respect to $P$ is the following (3-bounded) 2-complex $C$.
Let $H$ be the graph obtained from $G$ by contracting each class of $P$ to a single vertex and 
then removing some of the loops (and keeping parallel edges). 
The vertices of $C$ are the vertices of $H$ together with one extra vertex, which we call the 
\emph{top (of the cone)}. The edges of $C$ are the edges of $H$ together with one edge for each 
vertex $e$ 
of 
$G$ joining the top with the vertex of $H$ that corresponds to the partition class containing $e$. 
We have one face for every edge $f$ of $G$. If that edge is not an edge of $H$, its two 
endvertices in $G$ are in the same 
partition class; this face is only incident with the two edges of $C$ corresponding 
to 
these vertices. Otherwise the face is 
additionally 
incident with the edge of $H$ that corresponds to $f$. 
\begin{eg}
 The generalised cone construction has as a special case the cone construction; indeed we can 
just 
pick $P$ to consist only of singletons. However, this construction has more flexibility, for 
example 
if $G$ is connected and simple and $P$ just consists of a single vertex, the construction gives a 
2-complex 
with only two vertices 
such that $G$ is the link graph at both vertices. 
\end{eg}

\begin{lem}\label{reduce_to_looped_cone1}
Let $C$ be a 3-bounded 2-complex with a vertex $v$. 
If $C$ has no loop incident with $v$, then $C$ has a space minor that is a generalised cone whose 
link 
graph at the top is $L(v)$. 
\end{lem}

\begin{proof}
We obtain $C_1$ from $C$ by deleting all faces not incident with $v$. We obtain $C_2$ from $C_1$ by 
forgetting all incidences at the edges not incident with $v$. We obtain $C_3$ from $C_2$ by 
splitting 
all vertices different from $v$. It remains to prove the following. 
\begin{sublem}\label{is_cone}
 $C_3$ is a 
generalised cone over $L(v)$ with top $v$.
\end{sublem}
\begin{proof}
Let $w$ be a vertex of $C_3$ different from $v$. Since every face of $C_3$ has size two or three 
and is incident with $v$, there is an edge $e$ with endvertices $v$ and $w$. 
Let $P[w]$ be the set of those vertices $e'$ of $L(v)$ such that there is a path from $e$ to $e'$ 
all of whose edges are faces of size two in $C_3$ or else faces of size three in $C_3$ that contain 
a loop. By construction, every edge in $P[w]$ is 
incident with 
$w$. 
Any edge in the link graph $L(w)$ of $C_3$ with only one endvertex in $P[w]$ must be 
a face of $C_3$ of size three. Let $x$ be the endvertex of such an edge not in $P[w]$. Then $x$ is 
an edge of $C_3$ that is not incident with $v$; and thus is only incident with a single face; that 
is, $x$ has degree one in $L(w)$. Hence the connected component of $e$ is contained in $P[w]$ and 
these attached leaves. 
By our construction $L(w)$ is connected, so $P[w]$ is a connected subset of $L(w)$ and is equal to 
the set of edges of $C_3$ between $v$ 
and $w$.

It follows that $C_3$ is (isomorphic to) a 
generalised cone over the link graph $L(v)$ at the top $v$ with respect to the partition $(P[w]| 
w\in V(C_3)-v)$. 
\end{proof}
\end{proof}

\begin{lem}\label{gen_cone_deletion}
 Let $C$ be a generalised cone over a graph $L$, and let $f$ be an edge of $L$.
 Then there is a space minor of $C$ that is a generalised cone over $L-f$.
\end{lem}

\begin{proof}
We denote the top of $C$ by $v$. 
 We obtain $C_1$ from $C$ by deleting the face $f$. We obtain $C_2$ from $C_1$ by 
splitting\footnote{If the face $f$ has size three, this additional splitting is trivial and hence 
not necessary.} all vertices of $C_1$ that are incident with $f$ in $C$ except for $v$. It is 
straightforward to check that $C_2$ is a generalised cone over $L-f$.
\end{proof}

The following is immediate from the definition of generalised cones.

\begin{obs}\label{gen_cone_contr_outside}
 Let $C$ be a generalised cone and $e$ be an edge of $C$ not incident with the top.
 If $e$ is not a loop, then the space minor $C/e$ is a generalised cone (over the same graph).
\end{obs}

\begin{proof}
We denote the two endvertices of the edge $e$ by $x$ and $y$. The link graph $L(e)$ at $e$ in $C/e$ 
is obtained from the link graphs $L(x)$ and $L(y)$ of $C$ by gluing them together at the 
degree-one-vertex $e$, and then suppressing the vertex $e$. Thus the link graph $L(e)$ is connected. 
So $C/e$ is a generalised cone over the same graph as the generalised cone $C$.  
\end{proof}

\begin{lem}\label{gen_cone_contr}
 Let $C$ be a generalised cone over a graph $L$. 
 Assume $L$ contains an edge $f$ that has an endvertex $e$ of degree two.
 Then there is a space minor of $C$ that is (isomorphic to) a generalised cone over $L/f$.
\end{lem}

\begin{proof}
We denote the edge incident with the vertex $e$ aside from $f$ by $f'$. The edges $f$ and $f'$ are 
in series in the graph $L$.  
Now we consider $f$ and $f'$ as faces of $C$. 
If one of these faces has size two, we contract it and denote the resulting complex by $C'$. It is 
then straightforward to check that $C'$ is a generalised cone over $L/f$ or $L/f'$, 
respectively. As these two graphs are isomorphic, $C'$ is a generalised cone over $L/f$.

Hence we may assume that the two faces $f$ and $f'$ have size three. 
We denote the edge of $f$ not incident with the top $v$ of $C$ by $x$, and the edge of $f'$ not 
incident with $v$ by $x'$.

Recall that the edge $e$ is incident with the top $v$ of $C$. We denote the endvertex of $e$ 
aside from $v$ by $w$. In the link graph $L(w)$ at $w$, the vertex $e$ is only incident with the 
edges $f$ and $f'$, and the other endvertices of these edges have degree one. Hence the connected 
component 
of $L(w)$ containing $e$ is a path of length two. As $L(w)$ is connected by the definition of 
generalised cones, the link graph $L(w)$ must be equal to a path of length two with the vertices 
$e$, $x$ and $x'$. 
So the edge
$x$ is not a loop. 

Thus by \autoref{gen_cone_contr_outside} $C''=C/x$ 
is a generalised cone over $L$. In this generalised cone the face $f$ has degree two. Hence $C''/f$ 
is a generalised cone over the graph $L/f$. 
\end{proof}

\begin{lem}\label{gen_cone_subdiv}\label{gen_cone}
 Let $C$ be a generalised cone and $H$ be a subdivision of the link graph at the top that 
has no loops. 
Then $C$ 
has a space minor that is a generalised cone over $H$. 
\end{lem}

\begin{proof}
By \autoref{gen_cone_deletion}, there is a space 
minor $C_1$ of $C$ that is a generalised cone over a graph $L'$ so that $H$ can be obtained from 
$L'$ by suppressing vertices of degree two. By 
\autoref{gen_cone_contr}, there is a space minor 
$C_2$ of $C_1$ that is a generalised cone over the graph $H$.
\end{proof}

\vspace{.3 cm}

In the following we introduce `looped generalised cones' and prove for them analogues of 
\autoref{reduce_to_looped_cone1} and \autoref{gen_cone}. 

A \emph{looped generalised cone} is obtained from a generalised cone by attaching a loop at the top 
of the cone, adding some faces of size one only containing that loop and adding the incidence 
with the loop to some existing faces of size two. This is well-defined as all faces of a 
generalised cone are incident with the top. 
The following is proved analogously to \autoref{reduce_to_looped_cone1}\footnote{The statement 
analogue to \autoref{is_cone} is that `$C_3$ is a looped generalised cone over $L(v)$ with top 
$v$'. By the proof of that sublemma it follows that the 2-complex $C_3/e$, obtained from $C_3$ by 
contracting the loop $e$, is a generalised cone. Using the 
definition of looped generalised cone, it follows that $C_3$ is a 
looped generalised cone with the desired property. }.

\begin{lem}\label{reduce_to_looped_cone2}
Let $C$ be a 3-bounded 2-complex and let $v$ be a vertex. 
If $C$ has precisely one loop $e$ incident with $v$, then $C$ has a space minor 
that is a looped generalised cone whose link 
graph at the top is $L(v)$.  
\qed
\end{lem}

We prove the following analogue of \autoref{gen_cone} for looped generalised cones. 
Given a graph $L$ together with two specified vertices $v$ and $w$, a \emph{strict subdivision} of 
$(L,v,w)$ is obtained by successively deleting edges from $L$ or contracting edges that have an 
endvertex $x$ of degree two such that neither $x$ is equal to $v$ or $w$ nor $x$ is adjacent to $v$ 
or $w$. 
Given a looped generalised cone, we refer to the link graph at the top together with the two 
vertices of that link graph corresponding to the loop as the \emph{specific link graph at the top}.

\begin{lem}\label{loop_gen_cone}
Let $C$ be a looped generalised cone and let $(H,v,w)$ be a strict subdivision of the 
specific link graph at the top. Assume that $H$ has no loops. Then $C$ 
has a space minor that is a looped generalised cone such that $(H,v,w)$ is the specific link graph 
at the top. 
\end{lem}

\begin{proof}The proof of \autoref{loop_gen_cone} is analogous to that of \autoref{gen_cone_subdiv}.
Indeed, the analogous proof of that of \autoref{gen_cone_deletion} shows the following.

\begin{sublem}\label{gen_cone_deletion_loop}
 Let $C$ be a looped generalised cone over a graph $L$, and let $f$ be an edge of $L$.
 Then there is a space minor of $C$ that is a looped generalised cone over $L-f$.
\end{sublem}
 Similarly like \autoref{gen_cone_contr} we prove the following.
  
\begin{sublem}\label{gen_cone_contr_loop}
 Let $C$ be a looped generalised cone with $(L,v,w)$ as the specific graph at the top.
 Assume $L$ contains an edge $f$ that has an endvertex $e$ of degree two such that $e$ is neither 
equal to $v$ or $w$ nor adjacent to $v$ or $w$.
 Then there is a space minor of $C$ that is (isomorphic to) a looped generalised cone over $L/f$.
\end{sublem}
\begin{proof}
Take the proof of \autoref{gen_cone_contr} and replace `generalised cone' by `looped generalised 
cone'. The analogous statement of \autoref{gen_cone_contr_outside} for looped generalised cones is 
also 
true.
\end{proof}
Thus we can apply the proof of  \autoref{gen_cone_subdiv} to prove \autoref{loop_gen_cone}, where 
we refer to \autoref{gen_cone_deletion_loop} in place of \autoref{gen_cone_deletion} and to 
\autoref{gen_cone_contr_loop} in place of \autoref{gen_cone_contr}. 
\end{proof}

Let $\Zcal_1$ be the set of generalised cones over the graphs $K_5$ or $K_{3,3}$.
Let $\Zcal_2$ be the set of looped generalised cones such that some member of $\Ycal'$ is 
a strict marked graph associated to the link graph at the top. 
Let $\Zcal$ be the union of $\Zcal_1$ and $\Zcal_2$. 

\subsection{A Kuratowski theorem}

In this subsection we prove \autoref{kura_intro}. First we prove the following. 

\begin{thm}\label{rot_minor}
Let $C$ be a simply connected locally 3-connected 2-dimensional simplicial complex.
Then $C$ has a planar rotation system if and only if $C$ has no space minor from 
the finite list $\Zcal$.
\end{thm}

\begin{proof}
If $C$ has a planar rotation system, it cannot have a space minor in $\Zcal$. Indeed, every complex 
$Z$ in $\Zcal$ has a link graph that is not loop planar. Hence no $Z$ in $\Zcal$ has a planar 
rotation system by \autoref{planar_rot_TO_loc_planar}. Since by \autoref{rot_closed_down} the class 
of  2-complexes with planar rotation systems is closed under space minors, $C$ cannot have a space 
minor in $\Zcal$. 

Now conversely assume that the simplicial complex $C$ has no space minor in $\Zcal$. Suppose for a 
contradiction that $C$ has no planar rotation system. Then by \autoref{rot_system_exists}, 
there is a 3-bounded space minor $C'$ that is not locally planar, where $C'$ is either $C$, or for 
some (non-loop) edge $e$ the contraction $C/e$ or there is 
a (non-loop) chordless 
cycle $o$ of 
$C$ and some $e\in o$ such that $C'=C/(o-e)$. 
We distinguish two cases.

{\bf Case 1:} $C$ or $C/e$ are not locally planar. 
Since $C$ has no parallel edges or loops by assumption, here $C'$ has no loop. 
Hence $C'$ has a vertex $v$ such that the link graph $L(v)$ at $v$ is not planar. 
By \autoref{reduce_to_looped_cone1} $C'$ has a space minor that is a generalised cone such that the 
link graph at the top is $L(v)$. By Kuratowski's theorem, $L(v)$ has a subdivision isomorphic to 
$K_5$ or 
$K_{3,3}$. So by \autoref{gen_cone} $C'$ has a space minor that is a generalised cone over $K_5$ or 
$K_{3,3}$. So $C$ has a space minor in $\Zcal_1$, which is the desired contradiction. 

{\bf Case 2:} Not Case 1. So $C'=C/(o-e)$. 
Let $v$ be the vertex of $C'$ corresponding to $o-e$. 
Since we are not in Case 1, all link graphs at vertices of $C$ are loop planar. In particular, 
it must be the link graph at $v$ that is not loop planar. 
\begin{sublem}\label{lem100}
 If the link graph $L(v)$ of $C'$ is not planar, there is an edge $e'\in o-e$ such that the link 
graph at 
$e'$ in $C/e'$ is not planar. 
\end{sublem}
\begin{proof}
We prove the contrapositive. So assume that for every edge $e'\in o-e$ the link graph at $e'$ of 
$C/e'$ is planar. Since $C$ is locally 3-connected, the planar rotation systems of the link graphs 
$L(w)$ at the vertices $w$ of $o$ are unique up to reversing. By \autoref{locally_at_edge} these 
rotation systems are reverse or agree at any rotator of a vertex in $o-e$.

Note that $L(v)$ is the vertex sum of the link graphs $L(w)$ along the set $o-e$ of gluing 
vertices.
Thus by reversing some of these rotation systems if necessary, we can apply \autoref{sum_planar1} 
to build a planar rotation system of $L(v)$. In particular, $L(v)$ is planar.
\end{proof}

By \autoref{lem100} and since we are not in Case 1, the link graph $L(v)$ is planar but not loop 
planar.

Since $C$ has no loops and parallel edges and $o$ is chordless, in this case $C'$ can only have the 
loop $e$. 
Thus by \autoref{reduce_to_looped_cone2} $C'$ has a space minor $C''$ that is a looped generalised 
cone 
such that the link graph at the top is $L(v)$. 

Since $C$ is locally 3-connected by 
assumption and by \autoref{sum_3con} the link graph $L(v)$ is 3-connected. So by 
\autoref{loop_to_marked} there is a marked graph $\hat 
G$ associated to $L(v)$ that is not planar. Let $G'$ be a strict marked graph associated to $L(v)$ 
with marked graph $\hat G$. By \autoref{loc_pl_summary_strict} $G'$ has a strict 
marked minor $\hat Y=(Y, x,z, (a_i,b_i)|i=1,2,3; \iota)$ in $\Ycal'$, where $x$ and $z$ are the 
vertices in $L(v)$ 
corresponding to the loop at $v$. 
By the definition of strict subdivision, we have that $(Y, x,z)$ is a strict subdivision of the 
specific link graph $(L(v),x,z)$ at the top of $C''$. 
So by \autoref{loop_gen_cone} $C'$ has a space minor that is a looped generalised cone such that 
$\hat Y$ is a strict marked graph associated to the top. 
So $C$ has a space minor in $\Zcal_2$, which is the desired contradiction. 
\end{proof}

\begin{proof}[Proof of \autoref{kura_intro}.]
By \autoref{emb_to_rot} a simply connected simplicial complex is 
embeddable in \Sthree\ 
if and only if it has a planar rotation system. So \autoref{kura_intro} is implied by 
\autoref{rot_minor}. 
\end{proof}

\begin{proof}[Proof of \autoref{kura_intro_hom}.]   
 By a theorem of \cite{3space2}, \cite[\autoref*{combi_intro_extended}]{3space2} to be precise, a 
simplicial complex 
with $H_1(C,\Fbb_p)=0$ is embeddable if and 
only if it is simply connected and it has a planar rotation system. So \autoref{kura_intro_hom} is 
implied by 
\autoref{rot_minor}. 
\end{proof}

\section{Concluding remarks}

The proof of \autoref{kura_intro} yields that quite a few properties are equivalent. This is 
summarised in the following. 

\begin{thm}\label{main_detailed2}
  Let $C$ be a simply connected locally 3-connected 2-dimensional simplicial complex. The following 
are equivalent. 
\begin{enumerate}
 \item $C$ has an embedding in the 3-sphere;
 \item $C$ has an embedding in some oriented 3-manifold;
 \item  $C$ has a planar rotation system;
 \item  all contractions of $C$ are locally planar;
 \item no contraction has a link graph that has $K_5$ or $K_{3,3}$ as a minor or a marked minor of 
the 12 marked graphs in the list $\Ycal$ defined in \autoref{sec:marked};
 \item $C$ has no space minor from the finite list $\Zcal$ defined in \autoref{sec:gen_cone}.
\end{enumerate}
\end{thm}

\begin{proof}
The equivalence between 1, 2 and 3 is proved in \cite{3space2}.  The equivalence between 3 and 4 is 
proved in \autoref{planar_rot_TO_loc_planar}. The equivalence between 1 and 5 is 
\autoref{weak_kura}. Finally, the equivalence between 1 and 6 is \autoref{kura_intro}.  
\end{proof}

\autoref{kura_intro_hom} is a structural characterisation of which locally 3-connected 
2-dimensional simplicial complex $C$ whose first homology group is trivial embed in 3-space. Does 
this have algorithmic consequences? 
The methods of this paper give an algorithm that checks in linear\footnote{Linear in the 
number of faces of $C$.} time  whether a locally 3-connected 2-dimensional simplicial complex 
has a planar rotation system. For general 2-dimensional simplicial complex we obtain a quadratic 
algorithm, see \cite{3space5} for details. But how easy is it to check 
whether 
$C$ is simply connected? For simplicial complexes in general this is not decidable; indeed for 
every finite presentation of a group one can build a 2-dimensional simplicial complex that has that 
fundamental group. 
However, for simplicial complexes that embed in some (oriented) 3-manifold; that is, that have a 
planar 
rotation system, this problem is known as the sphere recognition problem. 
Recently it was shown that sphere recognition lies in NP \cite{{Iva01},{Sch11}} and co-NP 
assuming the generalised Riemann hypothesis \cite{{HeuZen16},{Zen16}}. It is an open question 
whether there is a polynomial time algorithm.

\section{Acknowledgement}

I thank Nathan Bowler and Reinhard Diestel for useful discussions on this topic.

\bibliographystyle{plain}
\bibliography{literatur}

\begin{thebibliography}{10}

\bibitem{Bin59}
R.~H. Bing.
\newblock An alternative proof that 3-manifolds can be triangulated.
\newblock {\em Ann. Math.(2)}, 69:37--65, 1959.

\bibitem{3space2}
Johannes Carmesin.
\newblock Embedding simply connected 2-complexes in 3-space -- {II}. {R}otation
  systems.
\newblock Preprint 2017, available at "https://arxiv.org/pdf/1709.04643.pdf".

\bibitem{3space3}
Johannes Carmesin.
\newblock Embedding simply connected 2-complexes in 3-space -- {III}.
  {C}onstraint minors.
\newblock Preprint 2017, available at "https://arxiv.org/pdf/1709.04645.pdf".

\bibitem{3space4}
Johannes Carmesin.
\newblock Embedding simply connected 2-complexes in 3-space -- {IV}. {D}ual
  matroids.
\newblock Preprint 2017, available at "https://arxiv.org/pdf/1709.04652.pdf".

\bibitem{3space5}
Johannes Carmesin.
\newblock Embedding simply connected 2-complexes in 3-space -- {V}. {A} refined
  {K}uratowski-type characterisation.
\newblock Preprint 2017, available at "https://arxiv.org/pdf/1709.04659.pdf".

\bibitem{s3np_hard}
Arnaud de~Mesmay, Yo'av Rieck, Eric Sedgwick, and Martin Tancer.
\newblock Embeddability in {$\Bbb R^3$} is {NP}-hard.
\newblock Preprint 2017, available at: "https://arxiv.org/pdf/1708.07734".

\bibitem{DiestelBookCurrent}
Reinhard Diestel.
\newblock {\em Graph {T}heory \emph{(5th edition)}}.
\newblock Springer-Verlag, 2016.
\newblock \\ Electronic edition available at:\\ {\small\tt
  http://diestel-graph-theory.com/index.html}.

\bibitem{solving_Rota}
Jim Geelen, Bert Gerards, and Geoff Whittle.
\newblock Solving {R}ota's conjecture.
\newblock {\em Notices Amer. Math. Soc.}, 61(7):736--743, 2014.

\bibitem{Hatcher3notes}
Hatcher.
\newblock Notes on basic 3-manifold topology.
\newblock available at "http://www.math.cornell.edu/~hatcher/3M/3Mfds.pdf".

\bibitem{HeuZen16}
Michael Heusner and Raphael Zentner.
\newblock A new algorithm for 3-sphere recognition.
\newblock Preprint 2016, available at arXiv:1610.04092.

\bibitem{Iva01}
S.~V. Ivanov.
\newblock Recognizing the 3-sphere.
\newblock {\em Illinois J. Math.}, 45(4):1073--1117, 2001.

\bibitem{lovasz_gm_survey}
L\'aszl\'o Lov\'asz.
\newblock Graph minor theory.
\newblock {\em Bull. Amer. Math. Soc. (N.S.)}, 43(1):75--86, 2006.

\bibitem{mstw_3sphere_decidable}
Jiri Matousek, Eric Sedgwick, Martin Tancer, and Uli Wagner.
\newblock Embeddability in the 3-sphere is decidable.
\newblock In {\em Computational geometry ({S}o{CG}'14)}, pages 78--84. ACM, New
  York, 2014.
\newblock Extended version available at "https://arxiv.org/pdf/1402.0815".

\bibitem{MoharThomassen}
B.~Mohar and C.~Thomassen.
\newblock {\em Graphs on Surfaces}.
\newblock Johns Hopkins, 2001.

\bibitem{moise}
Edwin~E. Moise.
\newblock {A}ffine structures in 3-manifolds. {V}. {T}he triangulation theorem
  and {H}auptvermutung.
\newblock {\em Annals of Mathematics. Second Series}, 56:96--114, 1952.

\bibitem{Oxley2}
James Oxley.
\newblock {\em {M}atroid {T}heory \emph{(2nd edition)}}.
\newblock Oxford University Press, 2011.

\bibitem{Pap43}
C.~Papakyriakopoulos.
\newblock A new proof for the invariance of the homology groups of a complex
  (in greek).
\newblock {\em Bull. Soc. Math. Grece}, 22:1--154, 1946.

\bibitem{linkless_emb_survey}
J.~L. Ram{\'\i}rez~Alfons{\'\i}n.
\newblock Knots and links in spatial graphs: a survey.
\newblock {\em Discrete Math.}, 302(1-3):225--242, 2005.

\bibitem{MR2099147}
Neil Robertson and P.~D. Seymour.
\newblock Graph minors. {XX}. {W}agner's conjecture.
\newblock {\em J. Combin. Theory Ser. B}, 92(2):325--357, 2004.

\bibitem{Sch11}
Saul Schleimer.
\newblock Sphere recognition lies in np.
\newblock {\em In Michael Usher, editor, Low- dimensional and Symplectic
  Topology. American Mathematical Society.}, 82:183--214, 2011.

\bibitem{Wagner_minor}
Uli Wagner.
\newblock Minors, embeddability, and extremal problems for hypergraphs.
\newblock pages 569--607, 2013.

\bibitem{Zen16}
Raphael Zentner.
\newblock Integer homology 3-spheres admit irreducible representations in
  {SL}(2,{C}).
\newblock Preprint 2016, available at arXiv:1605.08530.

\end{thebibliography}

\end{document}